\documentclass[twoside,11pt]{article}

\usepackage{jmlr2e}
\usepackage{amsmath}
\usepackage{bbm}
\DeclareGraphicsExtensions{.pdf}
\usepackage{float}
\usepackage[ruled,vlined]{algorithm2e}
\newtheorem{myalgorithm}[theorem]{Algorithm}
\newtheorem{sa}[theorem]{Summary of assumptions}

\firstpageno{1}

\usepackage{lastpage}
\jmlrheading{24}{2023}{1-\pageref{LastPage}}{9/22; Revised
6/23}{6/23}{22-1075}{Samuel N. Cohen, Deqing Jiang, Justin Sirignano}
\ShortHeadings{Neural Q-learning for solving PDEs}{Cohen, Jiang, Sirignano}

\begin{document}

\title{Neural Q-learning for solving PDEs}

\author{\name Samuel N. Cohen \email cohens@maths.ox.ac.uk  \\
       \addr Mathematical Institute, University of Oxford\\
       Oxford, OX2 6GG, UK
       \AND
       \name Deqing Jiang \email Deqing.Jiang@maths.ox.ac.uk \\
       \addr Mathematical Institute, University of Oxford\\
       Oxford, OX2 6GG, UK
       \AND
       \name Justin Sirignano \email sirignano@maths.ox.ac.uk\\
       \addr Mathematical Institute, University of Oxford\\
       Oxford, OX2 6GG, UK
       }

\editor{Michael Mahoney}

\maketitle

\begin{abstract}
Solving high-dimensional partial differential equations (PDEs) is a major challenge in scientific computing. We develop a new numerical method for solving elliptic-type PDEs by adapting the Q-learning algorithm in reinforcement learning. To solve PDEs with Dirichlet boundary condition, our ``Q-PDE" algorithm is mesh-free and therefore has the potential to overcome the curse of dimensionality. Using a neural tangent kernel (NTK) approach, we prove that the neural network approximator for the PDE solution, trained with the Q-PDE algorithm, converges to the trajectory of an infinite-dimensional ordinary differential equation (ODE) as the number of hidden units $\rightarrow \infty$. For monotone PDEs (i.e.\ those given by monotone operators, which may be nonlinear), despite the lack of a spectral gap in the NTK,  we then prove that the limit neural network, which satisfies the infinite-dimensional ODE, strongly converges in $L^2$ to the PDE solution as the training time $\rightarrow \infty$. More generally, we can prove that any fixed point of the wide-network limit for the Q-PDE algorithm is a solution of the PDE (not necessarily under the monotone condition). The numerical performance of the Q-PDE algorithm is studied for several elliptic PDEs. 
\end{abstract}
\vspace{5mm}

\section{Introduction}
High dimensional partial differential equations (PDEs) are widely used in many applications in physics, engineering, and finance. It is challenging to numerically solve high-dimensional PDEs, as traditional finite difference methods become computationally intractable due to the curse of dimensionality. In the past decade, deep learning has become a revolutionary tool for a number of different areas including image recognition, natural language processing, and scientific computing. The idea of solving PDEs with deep neural networks has been rapidly developed in recent years and has achieved promising performance in solving real-world problems (e.g. \cite{hu2022higher}, \cite{li2022deep}, \cite{cai2021physics}, and \cite{misyris2020physics}).

A number of deep-learning-based PDE solving algorithms, for instance the deep Galerkin method (DGM;  \cite{sirignano2018dgm}) and physics-informed neural networks (PINNs) (\cite{raissi2019physics}), have been proposed. Inspired by the Q-learning method in reinforcement learning, we propose a new ``Q-PDE" algorithm which approximates the solution of PDEs with an artificial neural network. In this paper, we prove that for monotone PDEs the Q-PDE training process gives an approximation which converges to the solution of a certain limit equation as the number of hidden units in a single-layer neural network goes to infinity. Furthermore, we prove that the limit neural network converges strongly to the solution of the PDE that the algorithm is trying to solve.  

Q-learning (\cite{watkins1992q}) is a well-known algorithm for computing the value function of the optimal policy in reinforcement learning. Deep Q-learning (DQN) (\cite{mnih2013playing}), a neural-network-based Q-learning algorithm, has successfully learned to play Atari games (\cite{mnih2015human}) and subsequently has become widely-used for many applications (\cite{zhao2019deep}). The Q-learning algorithm updates its value function approximator following a biased gradient estimate computed from input data samples. We propose an algorithm which is similar to Q-learning in the sense that we update the parameters of a neural network via a biased gradient flow in continuous time. As far as the authors are aware, this is the first time that a Q-learning algorithm has been developed to solve PDEs directly. 

The universal approximation theorem (\cite{cybenko1989approximation}) indicates that the family of single-layer neural networks are powerful function approximators. However, the universal approximation theorem only states that there exists a neural network which can approximate continuous functions arbitrarily well; it does not suggest how to identify the parameters for such a neural network. When the number of units in our neural network becomes large, it is, however, possible to obtain asymptotic limits, for example the ``neural tangent kernel'' limit (\cite{jacot2018neural}), which gives a variant of the law of large numbers for infinitely wide neural networks. We will use this approach to study the performance of the biased gradient flow as a training algorithm, and show that the wide-network limit satisfies an infinite-dimensional ordinary differential equation (ODE).

We apply our Q-PDE approach to second order nonlinear PDEs with Dirichlet boundary conditions. In particular, we are able to give strong convergence results when the differential operator satisfies a strong monotonicity condition. Monotone PDEs arise in various applications, particularly in PDEs arising from stochastic modeling -- the generators of ergodic stochastic processes are monotone (when evaluated with their stationary distributions), which suggests a variety of possible applications of our approach. Further, the subdifferentials of convex functionals (i.e. of maps from functions to the real line) are  monotone; this suggests that monotone PDEs may be a particularly well-suited class of equations for gradient methods, as they correspond to (a generalization of) minimizations of convex functionals. We will see that, given this monotonicity assumption, we can prove that the limit Q-PDE algorithm will converge (strongly in $L^2$) to the solution of the monotone PDE. More generally, we can prove that any fixed point of the wide-network limit for the Q-PDE algorithm is a solution of the PDE (not necessarily under the monotone condition).

In the remainder of the introduction, we provide a brief survey of relevant literature and some related approaches. The necessary properties of the PDEs which we study, the architecture of the neural networks we use, and the Q-PDE training algorithm we propose are presented in Section \ref{sec_2}. Section \ref{sec_preliminary} derives useful properties of the limiting system, while Section \ref{sec_converge} contains a rigorous proof that the training process converges to this limit. Analysis of the limit neural network, in particular the proof of convergence to the solution of the PDE, is presented in Section \ref{sec_lim}. Numerical results are presented in Section \ref{sec_num}.

 \subsection{Summary}
 In summary, this paper will provide the following:
\begin{itemize}
    \item A new neural network training algorithm for second order PDEs (of the form given by Assumption \ref{PDEAssumptions}), based on Q-learning (Algorithm \ref{neuralQlearningAlgo}), and designed to ensure Dirichlet boundary conditions are enforced.
    \item Proofs of various functional-analytic properties of the neural tangent kernel limit of this algorithm (Section \ref{sec_preliminary}).
    \item A rigorous convergence proof of the neural-tangent kernel limit for a wide single-layer network (Section \ref{sec_converge}). This limit gives a simple deterministic dynamical system characterizing the neural network approximation during training (Definition \ref{widenetworklimit}).
    \item Proof that any fixed point of the limit neural network training dynamics is a solution of the PDE (Theorem \ref{fixedpointsolution}).
    \item A strong convergence proof for the limit neural network as the training time $t \rightarrow \infty$, when considering a PDE given by a monotone operator (Section \ref{sec_lim}).
    \item Numerical examples to demonstrate effectiveness of the training algorithm (Section \ref{sec_num}).
\end{itemize}
A more precise summary of our mathematical results is given in Section \ref{mainresultssummary}.

\subsection{Relevant literature} \label{sec_review}
The idea of solving PDEs with artificial neural networks has been rapidly developed in recent years. Various approaches have been proposed: \cite{lagaris1998artificial},  \cite{lagaris2000neural},  \cite{lee1990neural}, and \cite{malek2006numerical} applied neural networks to solve differential equations on an a priori fixed mesh. However, the curse of dimensionality shows that these approaches cannot be extended to high dimensional cases. In contrast, as a meshfree method, DGM (\cite{sirignano2018dgm}) randomly samples data points in each training step. Based on this random grid, it derives an unbiased estimator of the error of the approximation using the differential operator and the boundary condition of the PDE and then iteratively updates the neural network parameters with stochastic gradient descent. This approach has been successful at solving certain high-dimensional PDEs including free-boundary PDEs, Hamilton--Jacobi--Bellman equations, and Burger's equation. Unlike DGM, our algorithm is inspired by Q-learning, in the sense that we compute a biased gradient estimator as the search direction for parameter training. The biased gradient estimator does not require taking a derivative of the differential operators of the neural network; only the gradient of the neural network model itself is required. The simplicity of this gradient estimator, together with the monotonicity of the PDEs we consider, allows a proof of the convergence of the Q-PDE training algorithm to the solution of the PDE. 

Another related area of research is PINNs (\cite{raissi2019physics}), which train neural networks to merge observed data with PDEs. This enables its users to make use of a priori knowledge of the governing equations in physics. Analysis of solving second order elliptic/parabolic PDEs with PINNs can be found in \cite{shin2020convergence}. DeepONet (\cite{lu2021learning}) proposes solving a family of PDEs in parallel. Its architecture is split into two parallel networks: a branch net, which approximates functions related to input PDEs, and a trunk net that maps spatial coordinates to possible functions. Combining these two networks, it is possible to learn the solution of a PDE given enough training data from traditional solvers. Mathematical analysis of DeepONet in detail can be found in \cite{lanthaler2021error}. DeepSets (\cite{germain2021deepsets}) are designed to solve a class of PDEs which are invariant to permutations. It computes simultaneously an approximation of the solution and its gradient. For solving fully nonlinear PDE, \cite{pham2021neural} estimates the solution and its gradient simultaneously by backward time induction. For linear PDEs, MOD-Net (\cite{zhang2021mod}) uses a DNN to parameterize the Green’s function and approximate the solution.

Further approaches, which go beyond looking for an approximation of the strong solution of the PDE directly, have also been studied. \cite{zang2020weak} proposes using generative adversarial networks (GANs) to solve a min-max optimization problem to approximate the weak solution of the PDE. The Fourier neural operator (FNO) (\cite{li2020fourier}) approach makes use of a Fourier transform and learns to solve the PDE via convolution. The solution of certain PDEs can be expressed in terms of stochastic processes using the Feynman--Kac formula. \cite{grohs2020deep} apply neural networks to learn the solution of a PDE via its probabilistic representation and numerically demonstrates the approach for Poisson equations. Based on an analogy between the BSDE and reinforcement learning, \cite{han2017deep} proposes an algorithm to solve parabolic semilinear PDEs and the corresponding backward stochastic differential equations (BSDEs), where the loss function
gives the error between the prescribed terminal condition and the solution of the
BSDE. PDE-net and its variants (\cite{long2018pde}, \cite{long2019pde}) attempt to solve inverse PDE problems: a neural network is trained to infer the governing PDE given physical data from sensors. In a recent paper \cite{sirignano2021pde}, neural network terms are introduced to optimize PDE-constrained models. The parameters in the PDE are trained using gradient descent, where the gradient is evaluated by an adjoint PDE. 
In  \cite{ito2021neural}, a class of numerical schemes for solving semilinear Hamilton--Jacobi--Bellman--Isaacs (HJBI) boundary value problems is proposed. By policy iteration, the semilinear problem is reduced into a sequence of linear Dirichlet problems. They are subsequently approximated by a feedforward neural network ansatz.
For a detailed overview of deep learning for solving PDEs, we refer the reader to \cite{E2021}, \cite{beck2020overview} and \cite{germain2021neural}, and references therein.

Many of these works do not directly study the process of training the neural network. Gradient descent-based training of neural networks has been mathematically analyzed using the neural tangent kernel (NTK) approach (\cite{jacot2018neural}). The NTK analysis characterizes the evolution of the neural network during training in the regime when the number of hidden units is large. For instance, \cite{jacot2018neural}, \cite{lee2017deep}, and \cite{lee2019wide} study the NTK limit for classical regression problems where a neural network is trained to predict target data given an input. Using an NTK approach, \cite{wang2020and} gives a possible explanation as to why sometimes PINNs fail to solve certain PDEs. Their analysis highlights the difficulty in approximating highly oscillatory functions using neural networks, due to the lack of a spectral gap in the NTK. The NTK approach is not without criticism, particularly when applied to deeper neural networks, where it has been seen to be unable to accurately describe observed performance (see, for example, \cite{ghorbani} or \cite{ChizatOB19}). Alternative approaches include mean-field analysis (e.g. \cite{pmlr-v99-mei19a}).

In this article, we train a neural network to solve a PDE, using a different training algorithm, for which we give a rigorous proof of the convergence to a NTK limiting regime (as the width of the neural network increases). This NTK approach yields a particularly simple dynamic under our training algorithm, allowing us to prove convergence to the solution of the PDE (as training time increases), despite the poorly behaved spectral properties of the NTK, for those PDEs given by monotone operators. Our Q-PDE algorithm is substantially different to the gradient descent algorithm studied in \cite{jacot2018neural}, \cite{lee2017deep}, \cite{lee2019wide}, since a PDE operator will be directly applied to the neural network to evaluate the parameter updates. 

\section{The Q-PDE Algorithm} \label{sec_2}
In this section, we first of all state our assumptions on the PDE and its domain. Then we describe the neural network approximator and our Q-PDE algorithm.
\subsection{Assumptions on the PDE}
We consider a class of second-order nonlinear PDEs on a bounded open domain $\Omega \subset \mathbb{R}^n$ with Dirichlet boundary conditions:\begin{align}
\begin{cases} \label{pde}
    \mathcal{L}u= 0\quad \text{in} \,\,  \Omega\\
    \,\,\,\,u= f \quad \text{on}\,\, \partial \Omega.
\end{cases}
\end{align}
We assume a measure
\footnote{For example, $\mu$ can be taken to be the Lebesgue measure on $\Omega\subset\mathbb{R}^n$, and the analytic theory we consider is essentially the same as that when using the Lebesgue measure (i.e. the Sobolev norms using $\mu$ are equivalent to those using $\mathrm{Leb}$). We allow for more general $\mu$, as this will roughly correspond to a choice of sampling scheme in our numerical algorithm, and gives us a degree of flexibility when considering the monotonicity of $\mathcal{L}$.}
$\mu$ is given on $\Omega$, with continuous Radon--Nikodym density $d\mu/d\,\mathrm{Leb}$ bounded away from $0$ and $\infty$ (where $\mathrm{Leb}$ denotes Lebesgue measure). It follows that $\mu(\Omega)<\infty$ and  $\mu(A)>0$ for all open sets $A\subset \Omega$; all integrals on $\Omega$ will be taken with respect to this measure (unless otherwise indicated).  

We will study strong Sobolev solutions to the PDE \eqref{pde}; that is, we are interested in solutions $u\in \mathcal{H}^2$, where
\begin{equation}
    \mathcal{H}^p = \bigg\{f\in L^2(\Omega, \mu): \|f\|_{\mathcal{H}^p}:=\Big( \sum_{|\alpha|\leq p} \|D^\alpha f\|_{L^2} \Big)<\infty\bigg\},
\end{equation}
where $Du$ is the weak derivative of $u$ and where the boundary condition $u|_{\partial \Omega}=f$ is understood in the usual trace-operator sense in $\mathcal{H}^1$.

For notational simplicity, we write $\mathcal{H}^2_{(0)}=\mathcal{H}^2\cap \mathcal{H}^1_0$ for the subspace of functions $u\in\mathcal{H}^2$ such that $u|_{\partial\Omega} = 0$, that is, with boundary trace zero. This remains a Hilbert space under the $\mathcal{H}^2$ inner product. We write $\|\cdot\|$ (without a subscript) for the $L^2$ norm and similarly for the inner product. We refer to \cite{evans2010partial} or \cite{adams2003sobolev} for further details and general theory of these spaces and concepts.

The class of PDEs we will consider are those for which $\mathcal{L}$, $f$ and $\Omega$ satisfy some (strong) regularity conditions, which we now detail. We first of all present the assumptions and lemmas for the domain and the boundary condition.
\begin{assumption}\label{boundaryRegular}
The boundary $\partial \Omega$ is $C^{3, \alpha}$, for some $\alpha\in(0,1)$, that is, three times continuously differentiable, with $\alpha$-H\"older continuous third derivative.
\end{assumption}
\begin{assumption}[Auxiliary function $\eta$]\label{auxiliaryfunction}
There exists a (known) function $\eta \in C^3_b(\mathbb{R}^n)$, which satisfies $0< \eta < 1$ in $\Omega$, and $\eta=0$ on $\partial \Omega$. Furthermore, its first order derivative does not vanish at the boundary (that is, for  $x\in \partial \Omega$ and $\mathbf{n}_{x}$ an outward unit normal vector at $x$, we have  $\langle \nabla \eta(x), \mathbf{n}_{x}\rangle \neq 0$).
\end{assumption}

Using Assumptions \ref{boundaryRegular} and \ref{auxiliaryfunction}, we have the following useful result, which allows us to approximate functions in $\mathcal{H}^2_{(0)}$. Its proof is given in
Appendix \ref{proof_214}.
\begin{lemma} \label{lemma_eta}
\begin{enumerate}
    \item The set of functions $C^3(\overline\Omega)\cap C_0(\overline\Omega)$ is dense in $\mathcal{H}^2_{(0)} = \mathcal{H}^2\cap \mathcal{H}^1_0$ (under the $\mathcal{H}^2$ topology).
    \item For any function $u \in C^3(\overline\Omega)\cap C_0(\overline\Omega)$, the function $\tilde{u}=u/\eta$ is in $C^2_b(\Omega)\subset \mathcal{H}^2$.
\end{enumerate}
\end{lemma}

In addition to these assumptions on $\mathcal{L}$, we make the following assumptions on the boundary value.
\begin{assumption}[Interpolation of the boundary condition function]\label{continuation}
There exists a (known) function $\bar f \in \mathcal{H}^2$ such that $\bar f|_{\partial\Omega}=f$ (or, more precisely, if $T:\mathcal{H}^1(\Omega)\to L^2(\partial\Omega)$ is the boundary trace operator, we have $T(\bar f)=f$). In the rest of this paper, for notational simplicity, we identify $f$ with its extension $\bar f$ defined on $\overline{\Omega}$.
\end{assumption}
\begin{remark}
The existence of some $\bar f$ satisfying Assumption \ref{continuation} is guaranteed, given a solution $u\in \mathcal{H}^2$ exists (as we could take $\bar f=u$). Practically, we assume not only that $\bar f$ exists, but that it is possible for us to use it as part of our numerical method to find $u$. In many cases, this is a mild and natural assumption, for instance when we have $f=0$, where we can simply take $\bar f=0$. 
\end{remark}
Now we present the assumptions of the PDE itself.
\begin{assumption}[Lipschitz continuity of $\mathcal{L}$] \label{assume_1}
There exists a constant $C$ such that for any $f_1, f_2 \in \mathcal{H}^2 \cap C^2$ and any $x \in \Omega$, the (nonlinear differential) operator $\mathcal{L}$ satisfies
\begin{align} \label{Lip}
    \begin{split}
    |\mathcal{L}f_1(x)-\mathcal{L}f_2(x)| &\leq C \bigg [ |f_1(x)-f_2(x)|+ \sum_{i=1}^n |\partial_{x_i}f_1(x)-\partial_{x_i}f_2(x)|\\&\quad + \frac{1}{2}\sum_{i,j =1}^n |\partial^2_{x_i x_j}f_1(x)-\partial^2_{x_i x_j}f_2(x)| \bigg ].
\end{split}
\end{align}
\end{assumption}
\begin{assumption}[Integrability of $\mathcal{L}$ at zero]\label{assume_LH2}
Taking $f_0\equiv 0$, we have $\mathcal{L}f_0 \in L^2$.
\end{assumption}
\begin{remark}
Under Assumptions \ref{assume_1} and \ref{assume_LH2}, $\mathcal{L}$ is a Lipschitz operator $\mathcal{L}:\mathcal{H}^2\to L^2$.
\end{remark}

Our algorithm will also make use of the following auxiliary function, which allows us to control behaviour of functions near the boundary of $\Omega$.

In order to prove convergence of our scheme for large training time,  we will particularly focus on PDEs given by strongly monotone operators (cf. \cite{browder1967existence}), which for convenience we take with the following sign convention.
\begin{assumption}[Strong $L^2$-monotonicity of $\mathcal{L}$] \label{Lmonotone}
There exists a constant $\gamma>0$ such that for any $f_1, f_2 \in \mathcal{H}^2_{(0)}$  the operator $\mathcal{L}$ satisfies 
\begin{align} \label{Monotone}
\langle f_1-f_2, \mathcal{L}f_1 - \mathcal{L}f_2\rangle \leq -\gamma \|f_1-f_2\|^2_{L^2},
\end{align}
where $\langle \cdot, \cdot\rangle$ is the $L^2(\mu)$ inner product. 
\end{assumption}

While these assumptions are somewhat restrictive, they are general enough to allow for the case $\mathcal{L}u = \mathcal{A}u - \gamma u + r$, for $r$ any $L^2$ function, where $\mathcal{A}$ is the generator of a sufficiently nice Feller process $X$ and $\gamma>0$, for an appropriate choice of sampling distribution $\mu$; see Appendix \ref{appMonotone} for further discussion.  In this case, the solution $u$ can be expressed (using the Feynman--Kac theorem) in terms of the stochastic process,
\begin{align}
    u(x) = \mathbb{E}\Big[e^{-\gamma \tau} f(X_\tau)+\int_0^\tau e^{-\gamma t}r(X_t)dt\Big|X_0=x\Big]
\end{align}
where $\tau = \inf\{t\ge 0: X_t\not\in \Omega\}$.

Further, these assumptions are also sufficiently general to allow some nonlinear PDEs of interest, for example Hamilton--Jacobi--Bellman equations under a Cordes condition, as discussed by  \cite{smears2014discontinuous}. The assumption of monotonicity is also connected to Lyapunov stability analysis, as  it is easy to check that strong monotonicity is equivalent to stating that  $v\mapsto \frac{1}{2}\|v\|_{L^2}^2$ is a Lyapunov function for the infinite-dimensional dynamical system  $dv/dt = \mathcal{L}v + \gamma v$. For more general discussion of monotone operators, and their connection to analysis of the traditional Galerkin method, we refer to \cite{zeidler2013nonlinear}.

Our final assumption on the PDE is that a solution exists.
\begin{assumption}[Existence of solutions]\label{assume_exist}
There exists a (unique) solution $u\in \mathcal{H}^2$ to \eqref{pde}.
\end{assumption}
\begin{remark}We note that, assuming strong monotonicity (Assumption \ref{Lmonotone}), if a solution exists then it is guaranteed to be unique: if $u,u'$ are two solutions, then $\mathcal{L}u = \mathcal{L}u'$, hence
$
        0=\langle u-u', \mathcal{L}u - \mathcal{L}u'\rangle\leq -\gamma \|u-u'\|^2$, 
    and so $u=u'$ almost everywhere.

    While existence of solutions in $\mathcal{H}^2$ is a strong assumption, it is often satisfied for weak solutions to elliptic equations, given the well-known elliptic regularity results which ensure solutions are sufficiently smooth.
    \end{remark}



\begin{remark}
The algorithm we present below can easily be extended to higher-order PDEs. To do this, further smoothness assumptions (on the activation function $\sigma$, continuation value $f$ and auxiliary function $\eta$) and higher moments of the initial weights $w^i$ in the neural network are needed, and one argument (Lemma \ref{lemma_eta}) needs to be extended using alternative approaches.  In this paper, for the sake of concreteness (and to avoid unduly long derivations), we present our results for PDEs up to second-order; the details of the extension to the general case are left as a tedious exercise for the reader.
\end{remark}

\begin{sa}\label{PDEAssumptions}
For ease of reference, we now summarize the assumptions we have made on our PDE:
\begin{enumerate}
    \item $\mathcal{L}$ is a nonlinear second-order differential operator, Lipschitz (as a map $\mathcal{H}^2(\Omega)\to L^2(\Omega)$), square integrable at zero (Assumptions \ref{assume_1}, \ref{assume_LH2}).
    \item The boundary value $f$ has a known continuation to a function in $\mathcal{H}^2(\Omega)$ (Assumption \ref{continuation}).
    \item $\Omega\subset\mathbb{R}^n$ is bounded and open, and  has a $C^{3,\alpha}$ boundary $\partial \Omega$,  which is the zero level set of a known auxiliary function $\eta\in C^3_b(\mathbb{R}^n)$, and $\eta$ has nonzero derivative on the boundary (Assumptions \ref{boundaryRegular}, \ref{auxiliaryfunction})
    \item $\mathcal{L}$ is strongly monotone, and there exists an $\mathcal{H}^2(\Omega)$ solution to the PDE (Assumptions  \ref{Lmonotone}, \ref{assume_exist}).
\end{enumerate}
We emphasise that only parts \textit{i}--\textit{iii} of this assumption will be used, except in Section \ref{sec_lim}, where the strong monotonicity and existence of solutions will also be needed.
\end{sa}

\subsection{Neural Network Configuration}
\label{sec_learning}
We will approximate the solution to the PDE \eqref{pde} with a neural network, which will be trained with a deep Q-learning inspired algorithm. In particular, we will study a single-layer neural network with $N$ hidden units:
\begin{align}
    S^N(x;\theta)=\frac{1}{N^\beta}\sum_{i=1}^N c^i\,\sigma(w^i \cdot x+b^i),
\end{align} where the scaling factor $\beta \in (\frac{1}{2},1)$ and $\sigma: \mathbb{R}\to \mathbb{R}$ is a non-linear scalar function. The parameters $\theta$ are defined as $\theta:=\{c^i, w^i, b^i\}_{i=1,...,N}$ where $c^i, b^i \in \mathbb{R}$, and $w^i \in \mathbb{R}^n$. The function $\eta$ can then be used to design a neural network model which automatically satisfies the boundary condition, introduced by \cite{mcfall2009artificial}:
\begin{align} \label{12}
    Q^N(x;\theta):=S^N(x;\theta) \eta(x) + f(x).
\end{align}

\begin{assumption}[Activation function]\label{assume_activation}
The activation function $\sigma\in C_b^4(\mathbb{R})$ is non-constant, where $C_b^k$ is the space of functions with $k$-th order continuous derivatives.
\end{assumption}
\begin{remark}
This assumption coincides with the condition of Theorem 4 of \cite{hornik1991approximation}, and guarantees that the functions $\sigma$ generate neural nets which are dense in the Sobolev space $\mathcal{H}^2$. (Hornik shows, under this condition, the stronger result of density in $\mathcal{H}^4$; we focus on $\mathcal{H}^2$, but require additional bounded derivatives as part of our proof.) We will particularly make use of this to establish Lemma \ref{dis_I}.
\end{remark}

Before training begins (i.e.~at $t = 0$), the neural network parameters are randomly initialized. The random initialization satisfies the assumptions described below. 

\begin{assumption}[Neural network initialization]\label{initialization}
The initialization of the parameters $\theta_0$, for all $i \in \{1,2,...,N\}$, satisfies: \begin{itemize}
    \item The parameters $c_0^i$, $w_0^i$, $b_0^i$ are independent random variables. 
    \item The random variables $c_0^i$ are bounded, $|c^i_0|<K_0$, and $\mathbb{E}[c^i_0]=0$.
    \item \label{cybenko}
The distribution of the random variables $w_0^i, b_0^i$ has full support. That is, for any open set $D \subseteq \mathbb{R}^{n+1}$, we have $\mathbb{P}((w_0^i,b_0^i)\in D)>0$.
\item For any indices $i,k\in\{1,...,N\}$, we have  $\mathbb{E}[|(w_0^i)_k|^3] < \infty$, $\mathbb{E}[|b_0^i|] < \infty$.
\end{itemize}
\end{assumption}

\subsection{Training Algorithm} \label{sec_algo}

We now present our algorithm for training the neural network $Q^N(x; \theta)$ to solve the PDE (\ref{pde}).  At training time $t$, the parameters of the neural network are denoted  $\theta_t=\{c_t^i, w_t^i, b_t^i\}_{i=1,...,N}$. For simplicity, we denote $S^N(\cdot;\theta_t)$ as $S_t^N$, and $Q^N(\cdot;\theta_t)$ as $Q_t^N$. Then, 

\begin{align}
    Q_t^N = S_t^N \eta + f.
\end{align}

Our goal is to design an algorithm to train the approximator $Q^N$ to find the solution of the PDE. One approach, see for instance \cite{sirignano2018dgm}, is to train the model $Q^N$ to minimize the average PDE residual:
\begin{align} \label{loss_untruncated}
    \int_{\Omega}[\mathcal{L}Q^N(x)]^2 d\mu(x).
\end{align} 
To improve integrability, we will smoothly truncate this objective function, and use the result to motivate our training algorithm. 

\begin{definition} [Smooth truncation function]
\label{psi}
The functions $\{\psi^N\}_{N\in \mathbb{N}}$ are a family of smooth truncation functions if, for some $\delta\in (0, \frac{1-\beta}{2})$, \begin{itemize}
    \item $\psi^N \in \mathbb{C}^2_b(\mathbb{R})$ is increasing on $\mathbb{R}$.
    \item $|\psi^N|$ is bounded by $2N^\delta$.
    \item $\psi^N(x)=x$ for $x \in [-N^\delta,N^\delta]$. \item $|(\psi^N)'| \leq 1$ on $\mathbb{R}$.
    \item $F^N:=(\psi^N) \cdot (\psi^{N})'$ is uniformly Lipschitz continuous on $\mathbb{R}$ for all $N \in \mathbb{N}$.
\end{itemize}
\end{definition}

For simplicity, we will usually describe the family $\psi^N$ as a (smooth) truncation function. A simple example of a truncation function is as follows: take $\delta = (1-\beta)/4>0$ and
\begin{align}
    \psi^N(x) = \int_0^x g^N(y)dy, \quad \text{where}\quad 
    g^N(x) = \begin{cases} 
    e^{-(|x|-N^\delta)^2}& \text{ if }|x|\geq N^\delta,\\
    1& \text{ if }|x|<N^\delta.
    \end{cases}
\end{align}

\begin{remark}\label{Ferrorbound}
It is easy to see that, for any smooth truncation function, the related function $F^N(x) = (\psi^N)\cdot(\psi^N)'$ satisfies
\begin{align}|F^N(x) - x|\leq 2|x|\mathbbm{1}_{\{|x|\geq N^{\delta}\}}\text{ for all } x\in \mathbb{R}.
\end{align}
\end{remark}

We will use a truncation function $\psi^N$ to modify (\ref{loss_untruncated}), and will consider minimizing the truncated objective:
\begin{align} \label{loss}
    \int_{\Omega}\big[\psi^N(\mathcal{L}Q^N(x))\big]^2 d\mu(x).
\end{align} 
To minimise a loss function like (\ref{loss}), one can apply gradient-descent based methods.

However, in practice, computing the gradient of $\mathcal{L}\big(Q_t^N(\cdot; \theta)\big)$ with respect to $\theta$ does not lead to an algorithm permitting simple analysis, as natural properties of the differential operator $\mathcal{L}$ are not directly preserved. Instead, we introduce a biased gradient estimator as follows:
\begin{definition}
The sequence of functions
\begin{align} \label{G}
    G^N(\theta_t)=\int_{\Omega} \bigg [ \psi^N(\mathcal{L}Q^N_t(x))\cdot({\psi^N}')(\mathcal{L}Q^N_t(x))\bigg ]\nabla_\theta (-Q_t^N(x))d\mu(x)
\end{align}
are called the \emph{biased gradient estimators} for the truncated loss. These can be approximated using Monte-Carlo sampling, as
\begin{align} \label{G2}
    \hat G^N_M(\theta_t)=\frac{1}{M}\sum_{i=1}^M \bigg [ \psi^N(\mathcal{L}Q^N_t(x_i))\cdot({\psi^N}')(\mathcal{L}Q^N_t(x_i))\bigg ]\nabla_\theta (-Q_t^N(x_i))
\end{align}
where $x_i$ are independent samples from the distribution $\mu$.
\end{definition}
Our analysis will focus on the biased gradient estimator $G^N$, however in our numerical implementation (in order to avoid the curse of dimensionality) we will use the Monte-Carlo estimates $\hat G^N_M$ for large $M$.

We shall see that this biased gradient is a continuous-space, PDE analog of the classic Q-learning algorithm. In summary, our ``Q-PDE" algorithm for solving PDEs is: 
\begin{myalgorithm}[Q-PDE Algorithm]\label{neuralQlearningAlgo}
We fix a family of smooth truncation functions $\{\psi^N\}_{N\in\mathbb{N}}$ and, for each value of $N\in \mathbb{N}$, we proceed as follows:
\begin{enumerate}
\item Randomly initialize the parameters $\theta_0$, as specified in Assumption \ref{initialization}. 
\item Train the neural network via the \emph{biased gradient flow}
\begin{align}
    \frac{d \theta_t}{dt}=-\alpha^N_t G^N(\theta_t). \label{Q-Learning}
\end{align}
with $G^N$ as in \eqref{G} and learning rate 
\begin{equation}\label{alphadef}\alpha^N_t=\alpha^N=\alpha N^{2\beta-1},
\end{equation}
where $\alpha$ is a positive constant. 
\end{enumerate}
The approximate solution to the PDE at training time $t$ is $Q_t^N$, as defined by \eqref{12}.
\end{myalgorithm}

In the remainder of this paper, we will study this Q-PDE algorithm for solving PDEs both mathematically and numerically.  First, we prove that the neural network model $Q^N_t$ converges to the solution of an infinite-dimensional ODE as the number of hidden units $N \rightarrow \infty$. Then, we study the limit ODE, and prove that it converges to the true solution of the PDE as the training time $t \rightarrow \infty$. Finally, in Section \ref{sec_num}, we demonstrate numerically that the algorithm performs well in practice with a finite number of hidden units. For high dimensional PDEs, it is impossible to form mesh-grids to apply conventional finite difference methods. However, our method is mesh-free. The crucial feature for implementing our algorithm in practice is to evaluate the integral term $G^N$. As discussed, we can do this using the Monte-Carlo estimate $\hat G^N_M$; therefore our algorithm has the potential to solve high-dimensional PDEs.

\begin{remark}
It is worth observing that, given this choice of biased gradient flow, there is generally no guarantee that our algorithm corresponds to stochastic gradient descent applied to any potential function. The key advantage of the use of this biased gradient is that it separates the derivatives in $\theta$ from the differential operator $\mathcal{L}$, which allows the underlying properties of the PDE to be preserved more simply. We will see that this results in particularly simple dynamics in the wide-network limit.
\end{remark}

\subsubsection{Main results}\label{mainresultssummary}
Our main mathematical results are the following:
\begin{theorem}[cf. Theorem \ref{converge_1}]
Define the kernel
\begin{align}\label{Bdefinition}
\begin{split}
    B(x,y)&=\eta(x)\eta(y)\mathbb{E}_{c,w,b}\Big[\sigma(w \cdot x+b)\sigma(w \cdot y+b)+c^2 \sigma'(w \cdot x+b)\sigma'(w \cdot y+b)(x \cdot y+1)\Big],
    \end{split}
\end{align}
where the expectation is taken with respect to the distribution of the random initialization of $c$, $w$, and $b$ satisfying Assumption \ref{initialization}, and corresponding integral operator $(\mathcal{B}v)(y)=\int_\Omega B(x,y)v(x)d\mu(x)$.
For a single-hidden layer neural network with $N$ units, as $N\to \infty$ the PDE approximator $Q^N$ trained using the Q-PDE algorithm converges (in $\mathcal{H}^2$, for each time $t$) to the deterministic $\mathcal{H}^2$-valued dynamical system
$    {dQ_t}/{dt} = \alpha \mathcal{B}\mathcal{L}(Q_t)$, with  $ Q_0 = f$,
and $Q_t|_{\partial\Omega} \equiv f$ for all $t$.
\end{theorem}
\begin{theorem} The limiting dynamical system $Q_t$ has the following convergence properties:
\begin{itemize}
    \item If $Q_t$ converges in $\mathcal{H}^1$ to a fixed point $Q$, then $Q$ is the solution to the PDE \eqref{pde}. (Theorem \ref{fixedpointsolution})
    \item If $\mathcal{L}$ is a monotone operator, and $u$ is the solution to the PDE \eqref{pde}, then for almost every sequence $t_k\to \infty$, we have $Q_{t_k}\to u$ in $L^2$.  (Theorem \ref{thm:Qconvergence})
    \item If $\mathcal{L}$ is a Gateaux differentiable monotone operator then, for almost every sequence $t_k\to \infty$, we have $\mathcal{B}\mathcal{L}Q_{t_k}\to 0$.  (Theorem \ref{thm:diffLconvergence})
\end{itemize}
\end{theorem}

\subsubsection{Similarity to Q-learning}
Our training algorithm's biased gradient flow is inspired by the classic Q-learning algorithm in a reinforcement learning setting. The Q-learning algorithm approximates the value function of the optimal policy for Markov decision problems (MDPs). It seeks to minimize the error of a parametric approximator, such as a neural network, $R(\cdot,\cdot;\theta)$:
\begin{align}
    L(\theta)= \sum_{(s,a)\in \mathcal{S}\times \mathcal{X}}[Y(s,a; \theta)-R(s,a;\theta)]^2 \pi(s,a).
\end{align}
Here $\mathcal{S}$ and $\mathcal{X}$ are the finite state and action spaces of the MDP and $\pi$ is a probability mass function which is strictly positive for every $(s,a)\in \mathcal{S}\times \mathcal{X}$. $Y$ is the target Bellman function 
\begin{align}
    Y(s,a; \theta)=r(s,a)+\gamma \sum_{s' \in \mathcal{S}} \max_{a' \in \mathcal{X}}R(s',a';\theta)p(s'|s,a).
\end{align}
Q-learning updates the parametric approximator $R$ via $\theta_{k+1}=\theta_k + \alpha_k^N U_k$ 
where $U$ is a biased estimator of the gradient of $(Y-R)^2$. Since the transition density $p(s'|s,a)$ is unknown, $Y-R$ is estimated using random samples from the Markov chain $(s_k, a_k)$: 
\begin{align}
    \Tilde{Y}(s_k,a_k)-R(s_k,a_k;\theta_k):= r(s_k,a_k)+ \gamma \max_{a' \in \mathcal{X}}R(s_{k+1,a'};\theta_k)-R(s_k,a_k;\theta_k).
    \label{StochasticSampleQLearning}
\end{align}
The Q-learning algorithm treats $Y$ (and $\tilde Y$) as a constant and takes a derivative only with respect to the last term in (\ref{StochasticSampleQLearning}). The Q-learning update direction $U_k$ 
is:
\begin{align}
 -U_k  =: \Big(\Tilde{Y}(s_k,a_k)-R(s_k,a_k;\theta_k)\Big)\nabla_{\theta_k} [-R(s_k,a_k;\theta_k)]
\end{align}
We emphasize that the term $\nabla_\theta \Tilde{Y}$ is not included, which means $U_k$ is a biased estimator for the direction of steepest descent.

Returning to the objective function for the Q-PDE algorithm, the gradient of the loss function \eqref{loss} is 
\begin{align}
    \int_{\Omega} \Big [ \psi^N(\mathcal{L}Q^N_t(x)){\psi^N}'(\mathcal{L}Q^N_t(x))\Big ]\nabla_\theta \mathcal{L}Q_t^N(x)d\mu(x). \label{true_gradient}
\end{align}
The update direction $G^N$ in \eqref{G} differs from \eqref{true_gradient}, as $\nabla_\theta \mathcal{L}Q_t^N(x)$ has been replaced with $\nabla_\theta (-Q_t^N(x))$. This is, conceptually, the same update direction as used in the Q-learning algorithm.

\section{Preliminary analysis of the training algorithm}\label{sec_preliminary}

We begin our analysis of the training algorithm (\ref{Q-Learning}) by proving several useful bounds, in particular, that the neural network parameters $c_t$ and $w_t$ are bounded in expectation. The proof is provided in Appendix \ref{proof_2}. We recall that $n$ is the dimension of our underlying space, that is, $\Omega\subset \mathbb{R}^n$.
\begin{lemma}[Boundedness of parameters] \label{bounded_t}
For all $T>0$, there exists a deterministic constant $C>0$ such that, for all $N \in \mathbb{N}$, all $0 \leq t \leq T$, all $i \in \{1,2,...N\}$, and all $k \in \{1,2,...n\}$, 
\begin{equation}
    |c_t^i|< C, \qquad \mathbb{E}|(w_t^i)_k|< C, \qquad\text{and}\qquad  \mathbb{E}|b_t^i| < C.
\end{equation}

\end{lemma}

\subsubsection{Dynamics of $Q_t^N$ }
We next consider the evolution of the output of the neural network $Q_t^N$ as a function of the training time $t$. Define $B_t^N$ as the symmetric, positive semi-definite kernel function
\begin{align}
    B^N_t(x,y)=\eta(x)\eta(y)A_t^N(x,y),
\end{align}
where $A_t^N$ is the \emph{neural tangent kernel} (\cite{jacot2018neural})
\begin{align}\begin{split}
     A^N_t(x,y)&=\Big(\nabla_\theta S^N_t(x)\Big)^\top \Big( \nabla_\theta S^N_t(y)\Big)\\
     &=\frac{1}{N}\sum_{i=1}^N\Big[\sigma(w_t^i\cdot  x +b_t^i)\sigma(w_t^i \cdot y +b_t^i)\\&\qquad +{c_t^i}^2\sigma'(w_t^i \cdot x +b_t^i)\sigma'(w_t^i \cdot y +b_t^i)(x\cdot y+1)\Big].
\end{split}
\end{align}

For any $y \in \overline{\Omega}$, we can derive the evolution of value $Q_t^N(y)$ using the chain rule:
\begin{align} \label{21}
    \frac{d Q_t^N(y)}{dt}= \nabla_{\theta} Q_t^N(y) \cdot \frac{d \theta_t}{dt}.
\end{align} 
The RHS of (\ref{21}) can be evaluated using (\ref{12}), (\ref{Q-Learning}), and Fubini's theorem, which yields
\begin{align} \label{37}
\begin{split}
    \frac{dQ^N_t(y)}{dt}
    &=\alpha \int_\Omega F^N\Big(\mathcal{L}Q_t^N(x)\Big)B^N_t(x,y)d\mu(x),
\end{split}
\end{align}
where $F^N$ is as in Definition \ref{psi}. 

\subsubsection{Limit kernel}
Equation \eqref{37} shows that the kernel $B^N$ has a key role in the dynamics of $Q^N$. We now characterize the limit of $B^N$ as $N\to \infty$. In Section \ref{sec31}, we will study this convergence in detail.

At time $t=0$, the parameters $c_0^i, w_0^i, b_0^i$ are independently sampled. Therefore, for each $(x,y)\in \Omega$, by the strong law of large numbers  $A_0^N(x,y)$ converges almost surely to $A(x,y)$, where \begin{align}
    A(x,y)=\mathbb{E}_{c,w,b}\big[\sigma(w \cdot x +b)\sigma(w \cdot y +b)+c^2\sigma'(w \cdot x +b)\sigma'(w \cdot y +b)(x \cdot y +1)\big].
\end{align} It follows that $B_0^N(x,y)$ converges almost surely to $B(x,y)$:
\begin{align}\label{eq:Bformdefn}
    \lim_{N \to \infty} B_0^N(x,y)= B(x,y):=\eta(x)\eta(y)A(x,y).
\end{align}
We can write \begin{align}
    B(x,y)=\eta(x)\eta(y)\mathbb{E}_{c,w,b}[\sigma(w\cdot x + b)\sigma(w\cdot y + b)+ c^2 U(x)\cdot U(y)],
\end{align}
where the ($w,b$-dependent) vector function $U$ is given by $U(x)= \sigma'(w \cdot x +b)\big(x+1/\sqrt{n}\big)$. As $\sigma\in C^4_b(\mathbb{R})$, we know that $A$, $A^N$, $B$ and $B^N$ are all uniformly continuous in $(x,y)$, so the above a.s.\ convergences hold for all $(x,y)\in \Omega$ simultaneously.

\begin{lemma} \label{M}
There exists a constant $M>0$ such that, for any $(x,y) \in \Bar{\Omega} \times \Bar{\Omega}$,
\begin{align}
    |B(x,y)|+\sum_k |\partial_{y_k} B(x,y)|+ \sum_{k,l} |\partial^2_{y_k,y_l} B(x,y)| \leq M.
\end{align}
\end{lemma}
\begin{proof}
   Taking $B$ as defined in \eqref{eq:Bformdefn}, as $\eta\in C^3_b$, $\sigma \in C^4_b$ and each term of $\{c_t^i\}_{i\in \mathbb{N}}$ is uniformly bounded, we have $|B(x,y)|<C$. The partial derivative of $B$ is
\begin{align}
    \begin{split}
        &\partial_{y_k}B(x,y)\\
        &=\eta(x)\big(\partial_{y_k}\eta(y)\big)A(x,y)+\eta(x)\eta(y)\big(\partial_{y_k}A(x,y)\big)\\
        &=\eta(x)\partial_{y_k}\eta(y)A(x,y)+\eta(x)\eta(y)\mathbb{E}_{c,w,b}\bigg [ \sigma(w \cdot x +b)\sigma'(w \cdot y +b) w_k \\
        &\quad + c^2 \sigma'(w \cdot x +b) [w_k \sigma''(w \cdot y+b )x \cdot y+\sigma'(w \cdot y+b) x_k]\bigg].
    \end{split}
\end{align} 
As before, our assumptions on $\eta$, $\sigma$ and $\{c_t^i\}_{i\in \mathbb{N}}$ guarantee that the terms in $\partial_{y_k}B(x,y)$ are uniformly bounded for all $x,y \in \Bar{\Omega}$, so $|\partial_{y_k}B(x,y)|<C$; similarly for $\partial^2_{y_k y_l}B(x,y)$. The result follows.
\end{proof}
Given we have a kernel $B$, it is natural to define the corresponding integral operator, $v\mapsto \int_\Omega B(x,y) v(x) d\mu(x)$. It will be convenient to define this on a slightly larger space than $L^2$, to account for the effect of the function $\eta$.
\begin{definition}
We define the function space $L^2_\eta:= \{ f: \Omega \to \mathbb{R}|\, \eta f \in L^2\}$, and observe (as $\eta$ is bounded) that $L^2\subset L^2_\eta$. 
\end{definition}
\begin{definition}
The operator $\mathcal{B}:L^{2}_\eta \to L^2$ is defined by $
    (\mathcal{B}v)(y)=\int_\Omega B(x,y)v(x)d\mu(x)$
with $B$ as in \eqref{eq:Bformdefn}.
\end{definition}
\begin{lemma}
For any $v\in L^{2}_\eta$, we have $\mathcal{B}v \in \mathcal{H}^2_{(0)}$, or more specifically, $\mathcal{B}v \in C^2_b(\overline\Omega)$ and $(\mathcal{B}v)(y)=0$ for all $y\in \partial \Omega$. 
\end{lemma}
\begin{proof}
From the definition of $\mathcal{B}v$, we have $
    (\mathcal{B}v)(y)=\eta(y)\int_\Omega A(x,y)\eta(x)v(x)d\mu(x)$. Lemma \ref{M} and the smoothness of $\sigma$ and $\eta$ clearly imply that $\mathcal{B}v\in C^2_b(\overline\Omega)\subset \mathcal{H}^2$. We know that $\eta(x)=0$ for $x\in \partial \Omega$, so it is clear that $(\mathcal{B}v)(y)=0$ for all $y\in \partial \Omega$. As $\mathcal{H}^1_0$ is the kernel (in $\mathcal{H}^1$) of the boundary trace operator, and for smooth functions the trace operator is simply evaluation of the function on the boundary, it follows that $\mathcal{B}v\in \mathcal{H}^1_0$. We conclude that $\mathcal{B}v\in \mathcal{H}^2_{(0)} = \mathcal{H}^2\cap \mathcal{H}^1_0$.
\end{proof}
\begin{lemma}\label{lem:Blip1}
The linear map $\mathcal{B}:L^2_\eta \to \mathcal{H}^2$ is Lipschitz, in particular, there exists $C>0$ such that, for any $v\in L^2_\eta$, we have $\|\mathcal{B}v\|_{\mathcal{H}^2}\leq C\|\eta\cdot v\|_{L^2}$.
\end{lemma}
\begin{proof}
By definition, with $D^\alpha$ denoting the differential operator with respect to the $y$ argument, $
        \|\mathcal{B}v\|_{\mathcal{H}^2}^2 = \sum_{|\alpha|\leq 2}\|D^ \alpha \mathcal{B}v\|_{L^2}^2$.
    By the boundedness of $\eta$, $A$ and their derivatives, there exists a constant $K$ such that for any multi-index $\alpha$, $|D^\alpha [\eta(y)A(x,y)]|<K$  for any $x,y\in \Omega$. Hence, exchanging the order of integration (with respect to $x$) and differentiation (with respect to $y$) by the dominated convergence theorem,
    \begin{align}
    \begin{split}
        \big|\big(D^\alpha\mathcal{B}v\big)(y)\big|&=\bigg | D^\alpha \bigg [\int_\Omega \eta(y)A(x,y)\eta(x)v(x)d\mu(x)\bigg ] \bigg |=\bigg |\int_\Omega D^\alpha [\eta(y)A(x,y)]\eta(x)v(x)d\mu(x)\bigg |\\
        &\leq \int_\Omega \bigg |D^\alpha [\eta(y)A(x,y)]\eta(x)v(x)\bigg|d\mu(x)\leq K \int_\Omega |\eta(x)v(x)|d\mu(x).
    \end{split}
    \end{align}
   By Jensen's inequality, there is a constant $C>0$ such that
    \begin{equation}
    \begin{split}
         \big|\big(D^\alpha\mathcal{B}v\big)(y)\big|^2 &\leq K^2\Big(\int_\Omega |\eta(x)v(x)| d\mu(x)\Big)^2 \leq C\int_\Omega |\eta(x)v(x)|^2 d\mu(x) = C\| \eta \cdot  v \|_{L^2}^2.
         \end{split}
    \end{equation}
    Integrating over $\Omega$, we see that $
            \|D^ \alpha \mathcal{B}v\|_{L^2}^2\leq \big(C\mu(\Omega)\big)\| \eta \cdot  v \|_{L^2}^2$. Summing over indices $\alpha$ yields the result.
\end{proof}
Our next challenge is to prove that the image of $\mathcal{B}$ is dense in $\mathcal{H}^2_{(0)}$. This will allow us to represent the solutions to our PDE in terms of $\mathcal{B}$, and is closely related to the universal approximation theorems of \cite{cybenko1989approximation}, \cite{hornik1991approximation}, and others. 

\begin{lemma} \label{dis_I}
For a function $g \in \mathcal{H}^2(\Omega)$, if the equality \begin{align} \label{dis_RMK_3}
   \int_\Omega \sum_{|\alpha| \leq 2 } D^\alpha \sigma(w\cdot x+b) D^\alpha g(x)d\mu(x) =0
\end{align}
holds for all $w \in \mathbb{R}^n$, $b \in \mathbb{R}$, then  $g\equiv 0$.
\end{lemma}
\begin{proof}
Since $\mu$ is a finite measure and $\sigma\in C^2_b$ (cf. Assumption \ref{assume_activation}) and non-constant, from Theorem 4 in \cite{hornik1991approximation} we have that the linear span of $\{\sigma(w\cdot x+b)\}_{w,b\in \mathbb{R}}$ is dense in $\mathcal{H}^2$. Hence, there is no nontrivial element of $\mathcal{H}^2$ orthogonal to $\sigma(w\cdot x+b)$ for all $w, b$. In other words, for $g \in \mathcal{H}^2$, if the inner product $\langle \sigma_{w,b}, g \rangle_{\mathcal{H}^2}=0$ for all $w,b$, then $g=0$, which is the stated result.
\end{proof}

\begin{lemma} \label{dis_III}
Define 
$S: \mathcal{H}^2 \to \mathcal{H}^2$ by $(Sh)(y):= \int_{\Omega} \sum_{|\alpha| \leq 2} D_x^\alpha h(x) D_x^\alpha A(x,y)d\mu(x)$.
Then $S h = 0$ if and only if $h=0$.
\end{lemma}
\begin{proof}
    We first observe that, as $\sigma\in C^4_b(\mathbb{R})$, $D^\alpha_xA(x,y)$ is clearly a $C^2_b$ function of $y$ for all $x$. Consider the inner product between $Sh$ and $h$ on $\mathcal{H}^2$: 
    \begin{align} \label{dis_31}
        \begin{split}
            &\langle S h, h \rangle_{\mathcal{H}^2}= \sum_{|\alpha| \leq 2} \langle D^\alpha S h, D^\alpha h \rangle_{L^2} \\
            &= \int_{\Omega^2} \sum_{|\alpha_1| \leq 2, |\alpha_2| \leq 2} D^{\alpha_1}_x h(x) D^{\alpha_2}_y h(y)\bigg ( D^{\alpha_1}_x D^{\alpha_2}_y A(x,y)\bigg ) d\mu(x) d\mu(y)\\
            &= \mathbb{E}_{c,w,b}\bigg [\int_{\Omega^2} \sum_{|\alpha_1| \leq 2, |\alpha_2| \leq 2} \bigg\{D^{\alpha_1}_x h(x) D^{\alpha_2}_y h(y)\\
            &\qquad \times\bigg ( D^{\alpha_1}_x D^{\alpha_2}_y \Big[\sigma(w\cdot x +b)\sigma(w\cdot y +b) +c^2 U(x)\cdot U(y)\Big]\bigg ) d\mu(x) d\mu(y)\bigg\} \bigg ].\\
        \end{split}
    \end{align}
    Notice that \begin{align} \label{dis_32}
        \begin{split}
        &0\leq \mathbb{E}_{c,w,b}\bigg [  \Big\| c \int_{\Omega} \sum_{|\alpha| \leq 2} D^{\alpha}_x h(x)  D^{\alpha}_x U(x) d\mu(x) \Big\| ^2 \bigg ]\\
            &\mathbb{E}_{c,w,b}\bigg [\int_{\Omega^2} \sum_{|\alpha_1| \leq 2, |\alpha_2| \leq 2}\bigg\{ D^{\alpha_1}_x h(x) D^{\alpha_2}_y h(y) \bigg( D^{\alpha_1}_x D^{\alpha_2}_y(c^2 U(x)\cdot U(y))\bigg ) d\mu(x) d\mu(y) \bigg\}\bigg ]
        \end{split}
    \end{align}
    Combining (\ref{dis_31}) with (\ref{dis_32}) we derive 
    \begin{align} \label{dis_33}
        \begin{split}
            &\langle S h, h \rangle_{\mathcal{H}^2} 
            \geq \mathbb{E}_{c,w,b}\bigg [\bigg (\int_{\Omega} \sum_{|\alpha| \leq 2} D^{\alpha}_x h(x)  D^{\alpha}_x \sigma(w\cdot x +b) d\mu(x) \bigg)^2 \bigg ]\geq 0.
        \end{split}
    \end{align}
By Lemma \ref{dis_I}, for $h \neq 0$, there exists $\epsilon >0$, $w^* \in \mathbb{R}^n$, $b^* \in \mathbb{R}$ such that
$0<\epsilon=\Big |\int_{\Omega} \sum_{|\alpha| \leq 2} D^{\alpha}_x h(x)  D^{\alpha}_x \sigma(w^*\cdot x +b^*) d\mu(x)\Big |$.
As this integral is continuous with respect to the parameters $w$, $b$, there exists $\delta >0$ such that for any $(w,b) \in R_\delta:=\{(w,b): \|w-w^*\| + |b-b^*| \leq \delta\}$, we know $
    \bigg |\int_{\Omega} \sum_{|\alpha| \leq 2} D^{\alpha}_x h(x)  D^{\alpha}_x \sigma(w\cdot x +b) d\mu(x)\bigg | \geq \frac{\epsilon}{2}$.
Therefore, from (\ref{dis_33}) we have 
\begin{align}
\begin{split}
    \langle S h, h \rangle_{\mathcal{H}^2} & \geq  \mathbb{E}_{c,w,b}\bigg [\bigg (\int_{\Omega} \sum_{|\alpha| \leq 2} D^{\alpha}_x h(x)  D^{\alpha}_x \sigma(w\cdot x +b) d\mu(x) \bigg)^2 \bigg ]\geq \mathbb{E}_{c,w,b} \Big[\frac{\epsilon^2}{4} \mathbbm{1}_{R_\delta} \Big]>0.
    \end{split}
\end{align}
\end{proof}
\begin{lemma}\label{thm:densityofimA}
 Define the operator $\mathcal{A}:L^2 \to \mathcal{H}^2$ by $(\mathcal{A}f)(y)=\int_\Omega A(x,y)f(x)d\mu(x)$.
 The image of $\mathcal{A}$ is dense in $\mathcal{H}^2$, that is, for any $u \in \mathcal{H}^2$, there exists a sequence $\{v_k\}_{k \in \mathbb{N}}$ in $L^{2}$ such that $\lim_{k \to \infty} \big \| \mathcal{A}v_k - u \big \|_{\mathcal{H}^2} = 0$.
\end{lemma}
\begin{proof}
By the smoothness of the kernel $A$, we have $\mathrm{im}(\mathcal{A}):=\{ \mathcal{A}v: v \in L^{2}\}\subset \mathcal{H}^2$.
    By definition, for any $h,g \in \mathcal{H}^2$, the inner product between $\mathcal{A}h$ and $g$ in $\mathcal{H}^2$ space is \begin{align}
        \begin{split}
            \langle \mathcal{A}g, h \rangle_{\mathcal{H}^2} &= \sum_{|\alpha|\leq 2}\langle D^\alpha \mathcal{A}g, D^\alpha h \rangle_{L^2}= \int_\Omega \sum_{|\alpha|\leq 2} D^\alpha_x h(x) \int_\Omega D^\alpha_x A(x,y)g(y)d\mu(y) d\mu(x)\\
            &= \int_\Omega g(y) \int_\Omega \sum_{|\alpha|\leq 2}D_x^\alpha h(x) D_x^\alpha A(x,y) d\mu(x) d\mu(y)= \big \langle g, S h \big \rangle_{L^2}.
        \end{split}
    \end{align}
    We introduce the adjoint operator of $\mathcal{A}$ on $\mathcal{H}^2$, denoted $\mathcal{A}^*$, and write \begin{align} \label{dense_1}
        \begin{split}
            \langle \mathcal{A}g, h \rangle_{\mathcal{H}^2}= \langle g, \mathcal{A}^* h \rangle_{\mathcal{H}^2} = \langle g, S h \rangle_{L^2}.
        \end{split}
    \end{align}
    By Lemma \ref{dis_III}, $S h=0$ if and only if $h=0$. Therefore, by setting $g=S h$ in (\ref{dense_1}), we have $\langle g, \mathcal{A}^* h \rangle_{\mathcal{H}^2}= \langle S h, S h \rangle_{L^2} >0$ for any non-zero $h$. Thus, $
        \mathrm{ker}(\mathcal{A}^*):=\{h \in \mathcal{H}^2: \mathcal{A}^* h=0 \}=\{0\}$.
    Write $\overline{\mathrm{im}(\mathcal{A})}$ for the closure in $\mathcal{H}^2$ of $\mathrm{im}(\mathcal{A})$. We recall that the ortho-complement in $\mathcal{H}^2$ of $\mathrm{im}(\mathcal{A})$ is $\mathrm{ker}(\mathcal{A}^*)$, so we can decompose the space $\mathcal{H}^2$ as $\mathcal{H}^2= \overline{\mathrm{im}(\mathcal{A})} \oplus \mathrm{ker}(\mathcal{A}^*)$.
    We know that $\mathrm{ker}(\mathcal{A}^*)=\{0\}$ in $\mathcal{H}^2$, and hence conclude that $\mathrm{im}(\mathcal{A})$ is dense in $\mathcal{H}^2$.
\end{proof}
\begin{theorem}[Density in $\mathcal{H}_{(0)}^2$]\label{thm:densityofimB}
The image of $\mathcal{B}$ is dense in $\mathcal{H}^2_{(0)}$, that is, for any $u \in \mathcal{H}^2_{(0)}$, there exists a sequence $\{v_k\}_{k \in \mathbb{N}}$ in $L^{2}_\eta$ such that $\lim_{k \to \infty} \| \mathcal{B}v_k - u \|_{\mathcal{H}^2} = 0$.
\end{theorem}
\begin{proof}
For a function $u \in C^3(\overline\Omega)\cap C_0(\overline\Omega)$, we define the function 
$\tilde{u}=u/\eta$. By Lemma \ref{lemma_eta}(\emph{ii}), $\tilde{u}\in C^2_b(\Omega)\subset \mathcal{H}^2$. Since the image of $\mathcal{A}$ is dense in $\mathcal{H}^2$, there exists a sequence $w_k\in \mathcal{H}^2$ such that $\|\mathcal{A}w_k-\tilde{u}\|_{\mathcal{H}^2} \to 0$. Consequently, the boundedness of $\eta$ and its derivatives shows $
\big\|\eta(\mathcal{A}w_k-\tilde{u})\big\|_{\mathcal{H}^2}=\big\|\mathcal{B}\big(w_k/\eta\big)-u\big\|_{\mathcal{H}^2} \to 0$.
Since $C^3(\overline\Omega)\cap C_0(\overline\Omega)$ is dense in $\mathcal{H}^2_{(0)}$ (under the $\mathcal{H}^2$ topology, Lemma \ref{lemma_eta}(\emph{i})), we conclude that the image of $\mathcal{B}$ is dense in $\mathcal{H}^2_{(0)}$.
\end{proof}
\begin{remark}
The above result shows that it is, in principle, possible to approximate the PDE solution $u$ using an infinite neural network, with the boundary condition enforced by $\eta$. It also justifies the boundary-matching method proposed by \cite{mcfall2009artificial}. In particular, as $u-(1-\eta)f\in \mathcal{H}^2_{(0)}$, there exists a sequence $v_k\in L^2_\eta$ such that $\mathcal{B}v_k + (1-\eta)f \to u$ in $\mathcal{H}^2$. Of course, whether a particular training algorithm yields such a sequence is a separate question.
\end{remark}

We collect the remaining key properties of $\mathcal{B}$ into the following lemma.
\begin{lemma}\label{lem:BLip}
The linear operator $\mathcal{B}:L^2_\eta\to \mathcal{H}^2$ has the following properties:
\begin{enumerate}
\item $\mathcal{B}$ is strictly positive definite, and induces a norm $\|\cdot\|_\mathcal{B}$ on $L^2_\eta$, given by $
     \|v \|_\mathcal{B}:= \sqrt{\langle \mathcal{B}v, v \rangle_{L^2}}$.
In particular, $\|v\|_\mathcal{B}<\infty$ for all $v\in L^2_\eta$.
\item There exists a constant $\lambda>0$ such that $\|\mathcal{B} v \|^2_{L^2} \leq \lambda \|v\|_{\mathcal{B}}^2$.
\end{enumerate}
\end{lemma}
\begin{proof}
For $v\in L^2_\eta$, let $g:=\eta v \in L^2$. By definition, \begin{align}
\begin{split}
    \langle v, \mathcal{B}v \rangle &= \langle v, \eta \mathcal{A}g \rangle=\langle g, \mathcal{A}g \rangle=\int_{\Omega^2} g(x)g(y)A(x,y)d\mu(x)d\mu(y)\\
    &=\int_{\Omega^2} g(x)g(y)\mathbb{E}_{c,w,b}\big[\sigma(w\cdot x+b)\sigma(w\cdot y+b)\big]d\mu(x)d\mu(y)\\
    &\qquad +\int_{\Omega^2} g(x)g(y)\mathbb{E}_{c,w,b}\big[c^2 U(x)\cdot U(y)\big]d\mu(x)d\mu(y)\\
    &=\mathbb{E}_{c,w,b}\bigg[ \bigg(\int_\Omega \sigma(w\cdot x+b)g(x)d\mu(x)\bigg )^2+\bigg \| c \int_\Omega U(x)g(x)d\mu(x) \bigg \|^2 \bigg ]\\
    & \geq \mathbb{E}_{c,w,b}\bigg[ \bigg(\int_\Omega \sigma(w\cdot x+b)g(x)d\mu(x)\bigg )^2 \bigg ] \geq 0.
\end{split}
\end{align}  
For $v\neq 0$, we know $g \neq 0$. With Assumption \ref{assume_activation}, by Theorem 5 in \cite{hornik1991approximation}, there exists $w^* \in \mathbb{R}^2$, $b^* \in \mathbb{R}$ such that $\big|\int_\Omega \sigma(w^* \cdot x+b^*)g(x)d\mu(x)\big|=\epsilon >0$. As the integral is continuous with respect to parameters $w,b$ there exists $\delta>0$ such that for any $(w,b)$ in the set $R_\delta:=\{(w,b): \|w-w^*\|+|b-b^*|< \delta\}$, we have $|\int_\Omega \sigma(w\cdot x+b)g(x)d\mu(x)|> \epsilon/2$. Therefore $    \langle v, \mathcal{B}v \rangle  \geq \mathbb{E}_{c,w,b}\big[ \big(\int_\Omega \sigma(w\cdot x+b)g(x)d\mu(x)\big )^2 \big ] \geq \mathbb{E}_{c,w,b}\big[\mathbbm{1}_{R_\delta}\frac{\epsilon^2}{4}\big] >0$; so $\mathcal{B}$ defines a norm as stated.

From the above calculations, we also see that $\mathcal{A}$ is a positive-definite Hilbert--Schmidt integral operator on $L^2$. Therefore, its eigenfunctions span the entire $L^2$ space and the spectral theorem applies, in particular $\mathcal{A}$ has nonnegative eigenvalues bounded above. Suppose $\lambda$ is the supremum of the eigenvalues of $\mathcal{A}$, we have $\|\mathcal{A}g\|_{L^2} \leq \lambda \langle g, \mathcal{A}g \rangle$. Thus, as $|\eta(x)|<1$, we know $
        \|\mathcal{B} v \|^2_{L^2}= \|\eta \mathcal{A} (\eta v) \|^2_{L^2}\leq \| \mathcal{A} (\eta v) \|^2_{L^2} \leq \lambda \langle \eta v, \mathcal{A}(\eta v)\rangle = \lambda \langle v, \mathcal{B}v\rangle =\lambda \|v\|_{\mathcal{B}}^2$.
\end{proof}

\section{Convergence to the limit ODE} \label{sec_converge}
In this section, we seek to understand the behavior of our approximator $Q_t^N$ when $N\to\infty$.
We prove that the pre-limit process, $\{Q^N_t(\cdot)\}_{0\leq t \leq T}$, converges to a limiting process $\{Q_t(\cdot)\}_{0\leq t \leq T }$, in an appropriate space of functions. In this section, we will only use Assumptions \ref{PDEAssumptions}(\textit{i}--\textit{iii}), in particular we will not need the assumption that $\mathcal{L}$ is monotone or that a solution to the PDE exists in $\mathcal{H}^2$. The challenge here is that our operator $\mathcal{L}$ is not generally the gradient of any potential function, and is an unbounded operator (in the $L^2$ or supremum norms), and so some care is needed.

In the following subsections, we first  bound the difference between the kernels $B^N_t(x,y)$ and $B(x,y)$ in a convenient sense. We then show that, as $N$ becomes large, the neural network approximator converges to the solution of an infinite dimensional ODE. Our main convergence result is Theorem \ref{converge_1}.

\subsection{Characterizing the difference between kernels}\label{sec31}
We characterize the (second order) difference between two kernels at $(x,y)$ by
\begin{align}
\begin{split}
    {H(A_1,A_2)(x,y)}&:= |A_1(x,y)-A_2(x,y)|
    +\sum_k |\partial_{y_k}A_1(x,y)-\partial_{y_k}A_2(x,y)|\\& \quad +\sum_{k,l}|\partial^2_{y_k y_l}A_1(x,y)-\partial^2_{y_k y_l}A_2(x,y)|.
\end{split}
\end{align}

Note that partial derivatives are only taken with respect to $y$ components. From 
Lemma \ref{M}, there exists a constant $M$ such that, for all $x,y \in \Omega$,
\begin{align} \label{bound_B}
    H(B,0)(x,y)\leq M.
\end{align}
Similarly, the (second order) difference between two smooth functions is characterized by
\begin{align} \label{48}
    \begin{split}
        {G(f_1,f_2)(x)}:&=|f_1(x)-f_2(x)|
    +\sum_k |\partial_{x_k}f_1(x)-\partial_{x_k}f_2(x)|\\&\qquad +\sum_{k,l}|\partial^2_{x_k x_l}f_1(x)-\partial^2_{x_k x_l}f_2(x)|.
    \end{split}
    \end{align}
\subsubsection{Difference between kernel $B_t^N(x,y)$ and $B_0^N(x,y)$}
In this subsection we characterize the difference between kernel $B^N_t$ and $B^N_0$. We denote the expectation taken with respect to randomized initialization by $\mathbb{E}_{c_0,w_0,b_0}[\,\cdot\,]$.
The proofs of the following lemmas are included in the appendix.
\begin{lemma} \label{ata0_lemma_1}
There exists $C>0$ such that, for all $i,j \in \{1, 2, ..., n\}$, all $0 \leq t \leq T$, and all $N>0$,
\begin{align}
|A_t^N(x,y)-A_0^N(x,y)|&\leq CN^{\delta+\beta-1}, \label{39} \\
    |\partial_{y_i}A_t^N(x,y)-\partial_{y_i} A_0^N(x,y)| &\leq C N^{\delta+\beta-1} \bigg [ 1+ \frac{\sum_{k=1}^N |(w_0^k)_i|}{N} \bigg ], \label{40}\\
|\partial^2_{y_i y_j}A_t^N(x,y)-\partial^2_{y_i y_j} A_0^N(x,y)|
&\leq C N^{\delta+\beta-2}\bigg ( \sum_{k=1}^N \big(1+|(w_0^k)_i|\big)\big(1+|(w_0^k)_j|\big)
     \bigg ). \label{41}
\end{align}
\end{lemma}
\begin{lemma}\label{lemma33} There exists $C>0$ such that, for all $0 \leq t \leq T$ and all $N>0$, the expected difference between   $A_t^N$ and $A_0^N$ satisfies
\begin{align}
\begin{split}
     \mathbb{E}_{c_0,w_0,b_0}\Big[\big|H(A_t^N, A_0^N)(x,y)\big|^2\Big]
     &\leq C N^{2(\delta+\beta-1)}.
\end{split}
\end{align} \label{error_ata0_expectation}
\end{lemma}
After characterizing the difference between $A_t^N$ and $A_0^N$, we estimate the difference between the kernels $B_t^N$ and $B_0^N$. 

\begin{lemma}\label{l21} There exists $C>0$ such that, for all $0 \leq t \leq T$ and all $N>0$
\begin{align}
    \mathbb{E}_{c_0,w_0,b_0}\bigg[{N^{2\delta} \int_{\Omega^2}\big|H(B_t^N,B_0^N)(x,y)\big|^2 d\mu(x)d\mu(y)}\bigg]  \leq C N^{2(2\delta + \beta -1)}.
\end{align}
\end{lemma}
\begin{proof}
    Since $B_t^N(x,y)=\eta(x)\eta(y)A_t^N(x,y)$, by the Cauchy--Schwarz inequality and the assumption $\eta\in C^3_b$, for some constant $C>0$,
    \begin{align}
    \begin{split} \label{50}
        \big|H(B_t^N,B_0^N)(x,y)\big|^2&\leq C \bigg [|B_t^N(x,y)-B_0^N(x,y)|^2 + \sum_{k} |\partial_{y_k}B_t^N(x,y)-\partial_{y_k}B_0^N(x,y)|^2\\
       &\quad + \sum_{k,l} |\partial_{y_k y_l} B_t^N(x,y)-\partial_{y_k y_l} B_0^N(x,y)|^2 \bigg ].
    \end{split}
    \end{align}
By the product rule, we have 
\begin{align}
        \partial_{y_i} B_t^N(x,y)&= \eta(x)A_t^N(x,y) \partial_{y_i}\eta(y)+ \eta(x)\eta(y)\partial_{y_i} A_t^N(x,y),\\
         \partial^2_{y_i y_j} B_t^N(x,y)&=\eta(x)A_t^N(x,y) \partial^2_{y_i y_j}\eta(y)+\eta(x) \partial_{y_j}A_t^N(x,y) \partial_{y_i}\eta(y) \nonumber     \\
         &\quad +\partial_{y_i}A_t^N(x,y) \partial_{y_j}\eta(y)+\eta(x)\eta(y)\partial^2_{y_i y_j} A_t^N(x,y).
\end{align}
As $\eta\in C^3_b$, there exists a constant $C>0$ independent of training time $t$, such that
\begin{align}
    |B_t^N(x,y)-B_0^N(x,y)|&\leq C |A_t^N(x,y)-A_0^N(x,y)|, \label{42}\\
|\partial_{y_i} B_t^N(x,y)-\partial_{y_i} B_0^N(x,y)| &\leq C \bigg [|A_t^N(x,y)-A_0^N(x,y)|+|\partial_{y_i}A_t^N(x,y)-\partial_{y_i}A_0^N(x,y)|\bigg ], \label{43}
\end{align}
and
\begin{align}
\begin{split}
&|\partial^2_{y_i y_j} B_t^N(x,y)-\partial^2_{y_i y_j} B_0^N(x,y)|\\&\leq C \Big [|A_t^N(x,y)-A_0^N(x,y)|+|\partial_{y_i}A_t^N(x,y)-\partial_{y_i}A_0^N(x,y)|\\
        &
    \quad +|\partial_{y_j}A_t^N(x,y)-\partial_{y_j}A_0^N(x,y)| + |\partial^2_{y_i y_j}A_t^N(x,y)-\partial^2_{y_i y_j}A_0^N(x,y)| \Big ]. \label{44}
    \end{split}
\end{align}
Substituting \eqref{42}, \eqref{43} and \eqref{44} in \eqref{50},
    \begin{align}
        \begin{split}
            \big|H(B_t^N,B_0^N)(x,y)\big|^2 &\leq C \bigg [ |A_t^N(x,y)-A_0^N(x,y)|^2+ \sum_{k} |\partial_{y_k}A_t^N(x,y)-\partial_{y_k}A_0^N(x,y)|^2\\
       &\quad + \sum_{k,l}|\partial^2_{y_k y_l}A_t^N(x,y)-\partial^2_{y_k y_l}A_0^N(x,y)|^2
       \bigg ].
        \end{split}
    \end{align}
    By Lemma \ref{error_ata0_expectation}, there exists a constant $C>0$ such that for all $N \in \mathbb{N}$ and any time $0\leq t\leq T$
$        \mathbb{E}_{c_0,w_0,b_0} \big[\big|H(B_t^N,B_0^N)\big|^2\big] \leq CN^{2(\delta+\beta-1)}.
$
Thus, by Tonelli's theorem we have \begin{align}
        \begin{split}
            \mathbb{E}_{c_0,w_0,b_0} \bigg[\int_{\Omega^2} \big|H(B_t^N,B_0^N)(x,y)\big|^2 d\mu(x)d\mu(y)\bigg] &\leq CN^{2(\delta+\beta-1)}.
        \end{split}
    \end{align}
    Multiplying by $N^{2\delta}$ on both sides concludes the result.
\end{proof}
\subsubsection{Difference between $B^N_0(x,y)$ and $B(x,y)$}
We now characterize the difference between the kernels $B^N_0$ and $B$. The law of large numbers implies  $B^N_0$ is converging to $B$ almost surely; here we obtain a bound on its speed of convergence. Again, the proof of the following lemma is given in the appendix.

\begin{lemma} \label{a0a_lemma} There exists $C>0$ such that, for any $j,k,l \in \{1,2,...,n\}$, $\forall$ $N>0$,
 \begin{align}
 \begin{split}
 \mathbb{E}_{c_0,w_0,b_0}\Big[ \big|A_0^N(x,y)-A(x,y)\big|^2\Big] &\leq \frac{C}{N},\\
     \mathbb{E}_{c_0,w_0,b_0}\Big[ \big|\partial_{y_i}A_0^N(x,y)-\partial_{y_i}A(x,y)\big|^2\Big] &\leq \frac{C}{N},\\ 
     \mathbb{E}_{c_0,w_0,b_0}\Big[\big|\partial^2_{y_i y_j}A_0^N(x,y)-\partial^2_{y_i y_j}A(x,y)\big|^2\Big] &\leq \frac{C}{N}.
 \end{split}
\end{align}
\end{lemma}
\begin{lemma}\label{l22}There exists $C>0$ such that, for all $N>0$,
\begin{align}
    \mathbb{E}_{c_0,w_0,b_0}\bigg [{N^{2\delta} \int_{\Omega^2}\big|H(B_0^N,B)(x,y)\big|^2 d\mu(x)d\mu(y)}\bigg ] \leq \frac{C}{N^{1-2\delta}}.
\end{align}
\end{lemma}
\begin{proof}
Similar to (\ref{50}), we derive \begin{align}
\begin{split}
    \big|H(B_0^N,B)(x,y)\big|^2 &\leq C \bigg [ |A_0^N(x,y)-A(x,y)|^2 + \sum_{k} |\partial_{y_k}A_0^N(x,y)-\partial_{y_k}A(x,y)|^2\\
       &\quad + \sum_{k,l} |\partial^2_{y_k y_l}A_0^N(x,y)-\partial^2_{y_k y_l}A(x,y)|^2 \bigg ].
\end{split}
\end{align}
Therefore, Lemma \ref{a0a_lemma} guarantees $
        \mathbb{E}_{c_0,w_0,b_0}\big[\big|H(B_0^N,B)(x,y)\big|^2\big] \leq {C}/{N}$.
By Tonelli's theorem, we have \begin{align}
        \begin{split}
    &\mathbb{E}_{c_0,w_0,b_0}\bigg [N^{2\delta} \int_{\Omega^2}\big| H(B_0^N,B)(x,y)\big|^2 d\mu(x)d\mu(y)\bigg]\\&=\int_{\Omega^2}N^{2\delta}\mathbb{E}_{c_0,w_0,b_0} \Big[ \big|H(B_0^N,B)(x,y)\big|^2\Big] d\mu(x)d\mu(y)\leq \frac{C}{N^{1-2\delta}}.
        \end{split}
    \end{align}
\end{proof}
\subsubsection{Difference between kernel $B^N_t(x,y)$ and $B(x,y)$}
Combining the results from the above two subsections, we provide one of our key lemmas.
\begin{lemma}\label{l2h}
The kernels $B_t^N$ and $B$ satisfy
\begin{align} 
   \lim_{N \to \infty}\mathbb{E}_{c_0,w_0,b_0} \bigg [N^{\delta} \int_0^T \int_{\Omega^2}\big|H(B_u^N,B)(x,y)\big|^2 d\mu(x)d\mu(y) du \bigg ] = 0.
\end{align}
\end{lemma}
\begin{proof}
    By the triangle inequality,
    \begin{align}
        H(B_u^N,B)(x,y) \leq H(B_u^N,B_0^N)(x,y)+H(B_0^N,B)(x,y).
    \end{align} Therefore,
    \begin{align}
        \big|H(B_u^N,B)(x,y)\big|^2\leq 2\big|H(B_u^N,B_0^N)(x,y)\big|^2+2\big|H(B_0^N,B)(x,y)\big|^2.
    \end{align}
    Combining Lemmas \ref{l21} and \ref{l22}, we have, for any $0 \leq t \leq T$, \begin{align}
        N^{2\delta}\mathbb{E}_{c_0,w_0,b_0}\bigg[{\int_{\Omega^2}\big|H(B_u^N,B)(x,y)\big|^2 d\mu(x)d\mu(y)}\bigg] \leq CN^{2(2\delta+\beta-1)} +C{N^{2\delta-1}}.
    \end{align} 
    Integrating with respect to time, 
    \begin{align}
    \begin{split}
        &\mathbb{E}_{c_0,w_0,b_0} \bigg [N^{\delta} \int_0^T \int_{\Omega^2}\big|H(B_u^N,B)(x,y)\big|^2 d\mu(x)d\mu(y) du \bigg ]\\& \leq CT (N^{2(2\delta+\beta-1)} + {N^{2\delta-1}}), 
        \end{split}
    \end{align}
    which converges to zero as $N \to \infty$.
\end{proof}
\subsection{Convergence of initial approximator $Q_0^N$ to ${Q}_0$}
In this subsection, we show that the randomly initialized approximator $Q_0^N$ converges to its limit, given by ${Q}_0:= f$ 
as the number of hidden units $N$ goes to infinity.
\begin{lemma} \label{lemma_ini}
The initial approximator $Q^N_0$ satisfies $\lim_{N \to \infty}\mathbb{E}_{c_0,w_0,b_0}\big[ \big\|Q_0^N-{Q}_0\big\|_{\mathcal{H}^2}^2 \big] = 0$.
\end{lemma}
\begin{proof}
   By definition, for any indices $k,l$, 
   \begin{align}
       \begin{split}
           Q^N_0(x)-{Q}_0(x)&=\frac{\eta(x)}{N^\beta}\sum_{i=1}^N c_0^i\,\sigma(w_0^i \cdot x+b_0^i),\\
           \partial_{x_k}(Q^N_0-{Q}_0)(x)&=\frac{\eta(x)}{N^\beta}\sum_{i=1}^N c_0^i\,\sigma'(w_0^i \cdot x+b_0^i)(w_0^i)_k +\frac{\partial_{x_k}\eta(x)}{N^\beta}\sum_{i=1}^N c_0^i\,\sigma(w_0^i \cdot x+b_0^i),
       \end{split}
   \end{align}
   and 
   \begin{align}
       \begin{split}
&\partial^2_{x_k x_l}(Q^N_0-{Q}_0)(x)\\& =\frac{\eta(x)}{N^\beta}\sum_{i=1}^N c_0^i\,\sigma''(w_0^i \cdot x+b_0^i)(w_0^i)_k(w_0^i)_l+\frac{\partial_{x_l}\eta(x)}{N^\beta}\sum_{i=1}^N c_0^i\,\sigma'(w_0^i \cdot x+b_0^i)(w_0^i)_k\\
           &\quad + \frac{\partial_{x_k}\eta(x)}{N^\beta}\sum_{i=1}^N c_0^i\,\sigma'(w_0^i \cdot x+b_0^i)(w_0^i)_l+\frac{\partial^2_{x_k x_l}\eta(x)}{N^\beta}\sum_{i=1}^N c_0^i\,\sigma(w_0^i \cdot x+b_0^i).
       \end{split}
   \end{align}
   As $\{c_0^i,w_0^i,b_0^i\}$ are independent for different $i\in \{1,2,...,N\}$, we have, for a constant $C>0$ which may vary from line to line,
   \begin{align}
       \begin{split}
           \mathbb{E}_{c_0,w_0,b_0}\Big[\big|Q^N_0(x)-{Q}_0(x)\big|^2\Big]&=\frac{\eta(x)^2}{N^{2\beta-1}} \mathbb{E}_{c,w,b}\big[c^2\sigma(w \cdot x+b)^2\big]\leq \frac{C}{N^{2\beta-1}},
        \end{split}
        \end{align}
        and
        \begin{align}
        \begin{split}
           &\mathbb{E}_{c_0,w_0,b_0}\Big[\big|\partial_{x_k}Q^N_0(x)-\partial_{x_k}{Q}_0(x)\big|^2\Big]\\&\leq \frac{2}{N^{2\beta-1}}\mathbb{E}_{c,w,b}\bigg [\eta(x)^2c^2\sigma'(w \cdot x+b)^2w_k^2 +(\partial_{x_k}\eta(x))^2 c^2 \sigma(w \cdot x+b)^2 \bigg ]\leq \frac{C}{N^{2\beta-1}}.
           \end{split}
   \end{align}
   Similarly, 
   \begin{align}\label{secondQderivatx}
       \begin{split}
           \mathbb{E}_{c_0,w_0,b_0}\Big[\big|\partial^2_{x_k x_l}Q^N_0(x)-\partial^2_{x_k x_l}{Q}_0(x)\big|^2\Big]
           \leq \frac{C}{N^{2\beta-1}}.
       \end{split}
   \end{align}
Summing over all indices, then integrating over $\Omega$, gives
$\mathbb{E}_{c_0,w_0,b_0} \Big [\big\|Q_0^N-{Q}_0\big\|_{\mathcal{H}^2}^2 \Big ] \leq CN^{1-2\beta}$.
As $\beta>1/2$, we have the desired convergence.
\end{proof}

\subsection{Convergence of the approximator $Q_t^N$ to ${Q}_t$} \label{sec_33}
\begin{definition}[Wide network limit]\label{widenetworklimit}
For $t\geq 0$, we define the \emph{wide network limit} ${Q}_t$ by the infinite-dimensional ODE
\begin{align} \label{ode}
    \frac{d{Q}_t}{dt}(y)=\alpha \int_{\Omega}\mathcal{L}({Q}_t)(x) B(x,y)d\mu(x), \qquad \text{for all } y \in \overline{\Omega},
\end{align}
with initial value ${Q}_0(y)=f(y)$, for all $y \in \overline{\Omega}$. Equivalently, we can write
\begin{align} \label{limit_ode}
    \frac{dQ_t}{dt} = \alpha \mathcal{B}\mathcal{L}(Q_t), \qquad Q_0 = f.
\end{align}
\end{definition}

\begin{theorem}\label{QODEunique}
The limit ODE (\ref{limit_ode}) admits a unique solution in $\mathcal{H}^2$.
\end{theorem}
\begin{proof}
From Lemma \ref{lem:Blip1}, $\mathcal{B}$ is uniformly Lipschitz as a map $L^2 \to \mathcal{H}^2$. By Assumption \ref{assume_1}, $\mathcal{L}$ is uniformly Lipschitz as a map $\mathcal{H}^2 \to L^2$. Hence, the operator $\mathcal{B}\mathcal{L}:\mathcal{H}^2\mapsto \mathcal{H}^2 $ is uniformly Lipschitz.
From the Picard--Lindel\"of theorem (see, for example, Theorem 2.2.1 in \cite{kolokoltsov2019differential}), we know that (\ref{limit_ode}) admits a unique solution with $Q_t \in \mathcal{H}^2$ for all $t\ge 0$. 
\end{proof}

We next prove that $Q^N_t$ converges to ${Q}_t$, justifying the name `wide network limit'.
\begin{theorem}[Convergence to the limit process ${Q}_t$]\label{converge_1}
For any fixed time $0 \leq t \leq T$,
\begin{align}
    \lim_{N \to \infty}\mathbb{E}_{c_0,w_0,b_0}\Big [ \big\|Q_t^N-{Q}_t\big\|_{\mathcal{H}^2} \Big]=0.
\end{align}
In particular, the boundary value of $Q_t$ is given by $Q_t|_{\partial\Omega} =f$.
\end{theorem}
\begin{proof} 
From the dynamics of $Q^N$ and ${Q}$ (\eqref{37} and \eqref{ode}) we have 
\begin{align}
    Q^N_t(y)-Q_0^N(y)&=\alpha\int_0^t \int_{\Omega}F^N(\mathcal{L}Q^N_u(x))B^N_u(x,y)d\mu(x) du, \label{qn}\\
    {Q}_t(y)-{Q}_0(y)&=\alpha\int_0^t \int_{\Omega}\mathcal{L}{Q}_u(x)B(x,y)d\mu(x) du. \label{qntilde}
\end{align}
Differentiate $Q^N_t(y)$ and ${Q}_t(y)$ with respect to $y_k$ to obtain

\begin{align}
       \partial_{y_k}Q^N_{t}(y)-\partial_{y_k}Q_{0}^N(y)&=\alpha\int_0^t \int_{\Omega}F^N(\mathcal{L}Q^N_u(x))\partial_{y_k}B^N_{u}(x,y)d\mu(x) du, \label{qny}\\
    \partial_{y_k}{Q}_{t}(y)-\partial_{y_k}{Q}_{0}(y)&=\alpha\int_0^t \int_{\Omega}\mathcal{L}{Q}_u(x)\partial_{y_k}B(x,y)d\mu(x) du. \label{qntildey}
\end{align}
Twice-differentiate $Q^N_t$ and ${Q}_t$ with respect to $y_k$ and $y_l$ to obtain:
\begin{align}
    \partial^2_{y_k y_l}Q^N_{t}(y)-\partial^2_{y_k y_l}Q_{0}^N(y)&=\alpha\int_0^t \int_{\Omega}F^N(\mathcal{L}Q^N_u(x))\partial^2_{y_k y_l}B^N_{u}(x,y)d\mu(x) du, \label{qnyy}\\
    \partial^2_{y_k y_l}{Q}_{t}(y)-\partial^2_{y_k y_l}{Q}_{0}(y)&=\alpha\int_0^t \int_{\Omega}\mathcal{L}{Q}_u(x)\partial^2_{y_k y_l}B(x,y)d\mu(x) du \label{qntildeyy}.
\end{align}
Subtracting (\ref{qn}) from (\ref{qntilde}), we have \begin{align}
    \begin{split} \label{d1}
        &|Q_t^N(y)-{Q}_t(y)|-|Q_0^N(y)-{Q}_0(y)|
        \\&\leq \alpha \int_0^t \bigg |\int_{\Omega}F(Q^N_u)(x)B^N_u(x,y)-\mathcal{L}{Q}_u(x)B(x,y)d\mu(x) \bigg | du
        \\
        & \leq \alpha \int_0^t \bigg |\int_{\Omega}F^N(\mathcal{L}Q^N_u(x))\Big(B^N_u(x,y)-B(x,y)\Big)d\mu(x) \bigg | du \\
        &\quad + \alpha \int_0^t \bigg |\int_{\Omega}\Big(F^N(\mathcal{L}Q^N_u(x))-F^N(\mathcal{L}{Q}_u(x)\Big)B(x,y)d\mu(x) \bigg | du\\
        &\quad + \alpha \int_0^t \bigg |\int_{\Omega}\Big(F^N(\mathcal{L}{Q}_u(x))-\mathcal{L}{Q}_u(x)\Big)B(x,y)d\mu(x) \bigg | du.
    \end{split}
\end{align}
Subtracting (\ref{qny}) from (\ref{qntildey}), we have \begin{align}
    \begin{split} \label{d2}
        &|\partial_{y_k}Q_{t}^N(y)-\partial_{y_k}{Q}_{t}(y)|-|\partial_{y_k} Q_0^N(y) - \partial_{y_k}{Q}_0(y)| \\
        &\leq \alpha \int_0^t \bigg |\int_{\Omega}F^N(\mathcal{L}Q^N_u(x))\Big(\partial_{y_k}B^N_{u}(x,y)-\partial_{y_k}B(x,y)(x,y)\Big)d\mu(x)\bigg | du \\
        &\quad + \alpha \int_0^t \bigg |\int_{\Omega}\Big(F^N(\mathcal{L}Q^N_u(x))-F^N(\mathcal{L}{Q}_u(x)\Big)\partial_{y_k}B(x,y)d\mu(x) \bigg | du\\
        &\quad + \alpha \int_0^t \bigg |\int_{\Omega}\Big(F^N(\mathcal{L}{Q}_u(x))-\mathcal{L}{Q}_u(x)\Big)\partial_{y_k}B(x,y)d\mu(x) \bigg | du.
    \end{split}
\end{align}
Similarly, subtracting (\ref{qnyy}) from (\ref{qntildeyy}), we have \begin{align}
    \begin{split} \label{d3}
        &|\partial^2_{y_k y_l}Q_{t}^N(y)-\partial^2_{y_k y_l}{Q}_{t}(y)|- |\partial^2_{y_k y_l} Q_0^N(y)-\partial^2_{y_k y_l}{Q}_0(y)|
        \\
        &\leq \alpha \int_0^t \bigg |\int_{\Omega}F^N(\mathcal{L}Q^N_u(x))\Big(\partial^2_{y_k y_l}B^N_{u}(x,y)-\partial^2_{y_k y_l}B(x,y)(x,y)\Big)d\mu(x)\bigg | du \\
        &\quad+ \alpha \int_0^t \bigg |\int_{\Omega}\Big(F^N(\mathcal{L}Q^N_u(x))-F^N(\mathcal{L}{Q}_u(x)\Big)\partial^2_{y_k y_l}B(x,y)d\mu(x) \bigg | du\\
        &\quad+ \alpha \int_0^t \bigg |\int_{\Omega}\Big(F^N(\mathcal{L}{Q}_u(x))-\mathcal{L}{Q}_u(x)\Big)\partial^2_{y_k y_l}B(x,y)d\mu(x) \bigg | du.
    \end{split}
\end{align}

\noindent As stated in Definition \ref{psi}, the function $F^N=(\psi^N)\cdot (\psi^{N})'$ satisfies a global Lipschitz condition. Therefore, for some $C>0$,
\begin{align}
\begin{split}
    &\Big|F^N(\mathcal{L}Q^N_t)(x)-F^N(\mathcal{L}{Q}_t)(x))\Big| <C\Big|\mathcal{L}Q^N_t(x)-\mathcal{L}{Q}_t(x)\Big|\\
    &\leq C\bigg ( |Q^N_t(x)-{Q}_t(x)|+
    \sum_k |\partial_{x_k}Q_t^N(x)-\partial_{x_k}{Q}_t(x)| +\sum_{k,l}|\partial^2_{x_k x_l}Q_t^N(x)-\partial^2_{x_k x_l}{Q}_t(x)| \bigg )\\
    &= C G(Q_t^N,{Q}_t)(x).
\end{split}
\end{align}
Summing (\ref{d1}), (\ref{d2}) and (\ref{d3}) over all indices $k$ and $l$, we have 
\begin{align}\begin{split} \label{step_1}
&G(Q_t^N,{Q}_t)(y)\\&=|Q_t^N(y)-{Q}_t(y)|
    +\sum_k |\partial_{x_k}Q_t^N(y)-\partial_{x_k}{Q}_t(y)| +\sum_{k,l}|\partial^2_{x_k x_l}Q_t^N(y)-\partial^2_{x_k x_l}{Q}_t(y)|\\
    &\leq G(Q_t^N,{Q}_0)(y)    +CN^{\delta} \int_0^t \int_\Omega H(B_u^N,B)(x,y) d\mu(x) du \\
    &\quad + C\int_0^t \int_\Omega  G(Q_u^N,{Q}_u)(x)H(B,0)(x,y) d\mu(x) du \\
    &\quad + C\int_0^t \int_\Omega  |F^N(\mathcal{L}{Q}_u(x))-\mathcal{L}{Q}_u(x)|H(B,0)(x,y) d\mu(x) du.
\end{split}
\end{align}
Here, the coefficient $N^\delta$ comes from the term $F^N$, as by Assumption \ref{psi}, $|F^N|=|(\psi^N)\cdot (\psi^N)'|\leq 2 N^\delta$. As mentioned in Remark \ref{Ferrorbound}, we also know $|F^N(x)-x|\leq 2|x|\mathbbm{1}_{\{|x|\geq N^\delta\}}$.  As $H(B,0)$ is uniformly bounded (Lemma \ref{M}), from inequality (\ref{step_1}), for some $C>0$ we have 
\begin{align}\begin{split} \label{step_2}
    G(Q_t^N,{Q}_t)(y)&\leq G(Q_t^N,{Q}_0)(y)    +CN^{\delta} \int_0^t \int_\Omega H(B_u^N,B)(x,y) d\mu(x) du \\
    &\quad + C\int_0^t \int_\Omega  G(Q_u^N,{Q}_u)(x)d\mu(x) du \\
    &\quad + C\int_0^t \int_\Omega  |\mathcal{L}{Q}_u(x)|\mathbbm{1}_{\{|\mathcal{L}{Q}_u(x)|\geq N^{\delta}\}} d\mu(x) du.
\end{split}
\end{align}
By multiplying $G(Q_t^N,{Q}_t)(y)$ in inequality (\ref{step_2}), integrating on both sides, and applying the Cauchy--Schwarz inequality, from (\ref{step_2}) we derive 
\begin{align}
    \begin{split}
    \int_\Omega \big|G(Q_t^N,{Q}_t)(y)\big|^2 d\mu(y)    &\leq C \bigg ({\int_\Omega \big|G(Q_t^N,{Q}_t)(y)\big|^2 d\mu(y)}\bigg )^{\frac{1}{2}} \Bigg [ \bigg( \int_\Omega \big|G(Q_0^N,{Q}_0)(y)\big|^2 d\mu(y) \bigg )^{\frac{1}{2}}
    \\&\quad + N^{\delta} \int_0^t \bigg ({\int_{\Omega^2}\big|H(B_u^N,B)(x,y)\big|^2 d\mu(x)d\mu(y)}\bigg )^{\frac{1}{2}} du \\
    &\quad +  \int_0^t {\int_\Omega \big|G(Q_u^N,{Q}_u)(x)\big|^2d\mu(x)}du \\
    &\quad +\int_0^t \int_\Omega |\mathcal{L}{Q}_u(x)|\mathbbm{1}_{\{|\mathcal{L}{Q}_u(x)|\geq N^{\delta}\}} d\mu(x)du \Bigg ].
    \end{split}
\end{align}
We write
\begin{align}\label{Jdef}
\begin{split}
    J^N&:= \bigg ( \int_\Omega \big|(Q_0^N,{Q}_0)(y)\big|^2 d\mu(y) \bigg )^{\frac{1}{2}}  + N^\delta \int_0^T \int_{\Omega^2}\big|H(B_u^N,B)(x,y)\big|^2 d\mu(x)d\mu(y) du\\
    &\quad +\int_0^T \int_\Omega |\mathcal{L}{Q}_u(x)|\mathbbm{1}_{\{|\mathcal{L}{Q}_u(x)|\geq N^{\delta}\}}d\mu(x)du.
\end{split}
\end{align}
Then the inequality (\ref{step_2}) can be formulated as
\begin{align} \label{step_3}
    \bigg ({\int_\Omega \big|G(Q_t^N,{Q}_t)(y)\big|^2 d\mu(y)}\bigg )^{\frac{1}{2}} \leq J^N + C \int_0^t \bigg ({\int_\Omega \big|G(Q_u^N,{Q}_u)(x)\big|^2d\mu(x)}\bigg )^{\frac{1}{2}}du.
\end{align}
Since 
\begin{align}
\begin{split}
    G(Q_t^N,{Q}_t)(y)&= |Q_t^N(y) - {Q}_t(y)| + \sum_k |\partial_{y_k}(Q_t^N - {Q}_t)(y)|+ \sum_{k,l}|\partial^2_{y_k, y_l}(Q_t^N - {Q}_t)(y)|,
\end{split}
\end{align}
by a simple quadratic-mean inequality we have 
\begin{align} \label{qm}
   \|Q_t^N - {Q}_t\|^2_{\mathcal{H}^2} \leq \int_\Omega \big|G(Q_t^N,{Q}_t)(y)\big|^2 d\mu(y) \leq (n^2+n+1)\|Q_t^N - {Q}_t\|^2_{\mathcal{H}^2}.
\end{align}
Consequently, from (\ref{step_3}) we derive 
\begin{align} \label{key_ineq}
     \|Q_t^N - {Q}_t\|_{\mathcal{H}^2}\leq J^N + C\int_0^t \|Q_u^N - {Q}_u\|_{\mathcal{H}^2}du.
\end{align}
By applying Gr\"onwall's inequality, (\ref{key_ineq}) yields the exponential bound
$    \int_0^t \|Q_u^N - {Q}_u\|_{\mathcal{H}^2}\, du\leq \frac{J^N}{C} (e^{Ct}-1)$.

It remains to bound $J^N$ in \eqref{Jdef}, as $N$ goes to infinity. By Lemma \ref{lemma_ini} and (\ref{qm}),  $\mathbb{E}_{c_0,w_0,b_0}\big[ \int_\Omega \big|G(Q_0^N,{Q}_0)(y)\big|^2 d\mu(y) \big] \to 0$. By the dominated convergence theorem, we obtain
$\int_0^T \int_\Omega |\mathcal{L}{Q}_u(x)|\mathbbm{1}_{\{|\mathcal{L}{Q}_u(x)|\geq N^{\delta}\}}d\mu(x)du \to 0$. Finally, by Lemma \ref{l2h}, we show that $\mathbb{E}_{c_0,w_0,b_0}\big[N^{\delta}\int_0^T  {\int_{\Omega^2}\big|H(B_u^N,B)(x,y)\big|^2 d\mu(x)d\mu(y)} du\big] \to 0$. 
Therefore, $\mathbb{E}_{c_0,w_0,b_0}[J^N] \to 0$ as $N\to \infty$  and consequently, from the above exponential bound,
\begin{align} \label{path_convergence}
    \lim_{N \to \infty}\mathbb{E}_{c_0,w_0,b_0}\Big[\int_0^T \|Q_u^N - {Q}_u\|_{\mathcal{H}^2} \,du\Big]=0.
\end{align}
Substituting this result back into inequality (\ref{key_ineq}) and taking the expectation with respect to the random initialization leads to
$\lim_{N \to \infty} \mathbb{E}_{c_0,w_0,b_0}\big [\|Q_t^N - {Q}_t\|_{\mathcal{H}^2}\big ]=0$ for any time $t \leq T$.

Finally, observe that $Q^N_t|_{\partial\Omega} =f$ by construction, so almost sure convergence in $\mathcal{H}^2$ (for a subsequence of $N$) and continuity of the boundary trace operator imply $Q_t|_{\partial\Omega}=f$.
\end{proof}

\section{Convergence for large training time}\label{sec_lim}
In the previous section, we showed that, as the neural network is made wider, its training process converges to a process that satisfies an infinite-dimensional ODE \eqref{limit_ode}. In particular, the wide-limit of the approximate solution has the dynamics $\frac{dQ_t}{dt} = \alpha \mathcal{B}\mathcal{L}(Q_t)$. This limit process can therefore be regarded as an approximation of the setting where we use a large (single-layer) neural network, and our problem (in the limiting setting) is transformed into the study of this infinite-dimensional dynamical system.  The following theorem is an easy consequence.
\begin{theorem}\label{fixedpointsolution}
If the wide-network limit converges in $\mathcal{H}^1$ to a fixed point $Q$, then $Q$ is a solution to the PDE \eqref{pde}.
\end{theorem}
\begin{proof}
For a fixed point $Q$, we know $\mathcal{B}\mathcal{L}(Q)=0$. From Lemma \ref{lem:BLip}, this implies that $\mathcal{L}(Q)=0$, so the PDE dynamics are satisfied. As $Q^N$ satisfies the boundary conditions $Q^N_t|_{\partial \Omega} = f$, and $Q^N_t\to Q_t$ in $\mathcal{H}^2$ for each $t$, and $Q_t\to Q$ in $\mathcal{H}^1$ we see that $Q$ must satisfy the boundary condition also (by $\mathcal{H}^1$-continuity of the boundary trace operator). 
\end{proof}

In this section, we will consider the simple case where $\mathcal{L}$ is a monotone operator. This implies that the dynamical system $dv/dt = \mathcal{L}v$ converges in $L^2$ exponentially quickly, which suggests the dynamics of $Q$ will be well behaved. However, the presence of the operator $\mathcal{B}$ leads to some difficulties in our analysis. Nevertheless, we will show that, in this setting, our approximation $Q_t$ converges to the true solution $u$ of the PDE, at least for a generic subsequence of times. 

We assume $\alpha=1$ in \eqref{alphadef} for notational simplicity in this section, without loss of generality (as this is a simple rescaling of time).

\begin{theorem}\label{thm:Qconvergence}
Let $u$ be the solution to the PDE \eqref{pde} under our assumptions (Assumption \ref{PDEAssumptions}, now including part \textit{iv}), and assume the neural network configuration satisfies Assumptions \ref{assume_activation} and \ref{initialization}. Then the wide network limit $\{Q_t\}_{t\ge 0}$ of the Q-PDE algorithm (Definitions \ref{neuralQlearningAlgo}, \ref{widenetworklimit}) satisfies \begin{equation}
    \frac{1}{t}\int_0^t\|Q_s-u\|^2_{L^2} ds\to 0 \text{ as }t\to \infty.
\end{equation}
In particular, there exists a set $A\subset[0, \infty)$ with $\lim_{t\to \infty}\big(t^{-1}\mathrm{Leb}(A\cap[0,t])\big)= 0$ such that, for all sequences $t_k\to \infty$ which do not take values in $A$, we have the (strong) $L^2$ convergence 
\[\|Q_{t_k}-u\|_{L^2} \to 0.\]
\end{theorem}
\begin{proof} We begin by supposing we have a decomposition
\begin{equation}\label{Qspatialdecomposition}
    Q_0 = f = u + \mathcal{B}v_0 + g,
\end{equation}
for some $v_0\in L^2_\eta$ and {$g\in \mathcal{H}^2$}. As {$u\in \mathcal{H}^2$}, we can choose $v_0$ arbitrarily and then find $g$ using \eqref{Qspatialdecomposition}. Consider the process $v$ solving the ODE
    \begin{equation}\frac{dv_t}{dt} = \mathcal{L}(\mathcal{B}v_t + u +g)\end{equation}
    with initial value $v_0$ as in \eqref{Qspatialdecomposition}. As detailed in Appendix \ref{lem:ODEexist}, as $\mathcal{B}$ and $\mathcal{L}$ are both Lipschitz continuous in appropriate spaces, a standard Picard--Lindel\"of argument shows that  $v$ is uniquely defined in $L^2_\eta$, for all $t\ge 0$.
    
    By differentiating $\mathcal{B}v_t$ (and using dominated convergence to move the derivative through $\mathcal{B})$, we obtain 
    \begin{equation}\frac{d(\mathcal{B}v_t+u+g)}{dt} = \frac{d(\mathcal{B}v_t)}{dt} =  \mathcal{B}\mathcal{L}(\mathcal{B}v_t + u + g); \qquad \mathcal{B}v_0+u+g = Q_0.\end{equation}
    In particular, as the ODE defining $Q$ has a unique solution in $\mathcal{H}^2$ (Theorem \ref{QODEunique}), we have the identity $Q_t = \mathcal{B}v_t +u+g$.
    
    Using the Young inequality $2\langle x,y\rangle \leq \gamma \|x\|^2 + (4/\gamma) \|y\|^2$ for $\gamma >0$, together with the fact $\mathcal{L}(u)=0$ and our assumption that $\mathcal{L}$ is strongly monotone and Lipschitz, for some constant $k>0$ we have
    \begin{align}
        \begin{split}
 \frac{d}{dt}\|v_t\|^2_\mathcal{B} &= \frac{d}{dt} \langle \mathcal{B}v_t, v_t\rangle = 2 \langle \mathcal{B}v_t, \mathcal{L}(\mathcal{B}v_t +u+g)\rangle\\
    &= 2 \langle \mathcal{B}v_t, \mathcal{L}(\mathcal{B}v_t +u+g) - \mathcal{L}(u+g)\rangle + 2 \langle \mathcal{B}v_t, \mathcal{L}(u+g) - \mathcal{L}(u)\rangle\\
 &\leq -2\gamma \|\mathcal{B}v_t\|^2_{L^2} + \gamma\|\mathcal{B}v_t\|^2_{L^2} + \frac{4}{\gamma}\|\mathcal{L}(u+g) - \mathcal{L}(u)\|^2_{L^2}\\
    &\leq -\gamma \|\mathcal{B}v_t\|^2_{L^2} + \gamma k\|g\|^2_{\mathcal{H}^2}.
        \end{split}
    \end{align}
 From Lemma \ref{lem:BLip} there exists $\lambda>0$ such that, for all $v\in L^2_\eta$, $\|\mathcal{B} v_t \|^2_{L^2} \leq \lambda \|v_t\|_{\mathcal{B}}$.
     Using this value of $\lambda$, as $\|g\|^2_{\mathcal{H}^2}$ is a constant,
    $\frac{d}{dt} \big(\lambda\|v_t\|^2_\mathcal{B} - k \|g\|^2_{\mathcal{H}^2}\big) \leq -\gamma\lambda\big(\|\mathcal{B}v_t\|^2_{L^2} - k \|g\|^2_{\mathcal{H}^2}\big)$.    Using the definition of $\lambda$ and integrating,
    \begin{align}
    \begin{split}
        \Big(\|\mathcal{B}v_t\|^2_{L^2} - k \|g\|^2_{\mathcal{H}^2}\Big) &\leq \Big(\lambda\|v_t\|^2_\mathcal{B} - k \|g\|^2_{\mathcal{H}^2}\Big) \\
        &\leq \Big(\lambda\|v_0\|^2_\mathcal{B} - k \|g\|^2_{\mathcal{H}^2}\Big) -\gamma\lambda\int_0^t \Big(\|\mathcal{B}v_s\|^2_{L^2} - k \|g\|^2_{\mathcal{H}^2}\Big)ds.
        \end{split}
    \end{align}
    Writing $h_t = \int_0^t \big(\|\mathcal{B}v_s\|^2_{L^2} - k \|g\|^2_{\mathcal{H}^2}\big)ds$, we see 
    ${dh_t}/{dt} \leq (\lambda\|v_0\|^2_\mathcal{B} - k \|g\|^2_{\mathcal{H}^2}) -\gamma\lambda h_t$,
    and hence Gr\"onwall's inequality yields
    \begin{equation}h_t = \int_0^t \Big(\|\mathcal{B}v_s\|^2_{L^2} - k \|g\|^2_{\mathcal{H}^2}\Big)ds \leq (\lambda\|v_0\|^2_\mathcal{B} - k \|g\|^2_{\mathcal{H}^2})\int_0^t e^{- \gamma\lambda(t-s)}ds.\end{equation}
    Recalling that $Q_t-u = \mathcal{B}v_t + g$ and \begin{align}
        \|\mathcal{B}v+g\|^2_{L^2}  \leq 2\|\mathcal{B}v\|^2_{L^2} + 2\|g\|^2_{\mathcal{H}^2} = 2(\|\mathcal{B}v\|^2_{L^2} - k\|g\|^2_{\mathcal{H}^2}) + (2k+2)\|g\|^2_{\mathcal{H}^2},
    \end{align}
    we conclude that, for some constant $C>0$,
    \begin{equation}\label{Qsquarebound}
    \int_0^t\|Q_s-u\|^2_{L^2} ds = \int_0^t\|\mathcal{B}v_s+g\|^2_{L^2}ds \leq Ct\|g\|_{\mathcal{H}^2} + C\|v_0\|^2_{\mathcal{B}}.
    \end{equation}

This inequality must hold for all choices of $g$ and $v_0$ satisfying \eqref{Qspatialdecomposition}, and we can choose them to optimize \eqref{Qsquarebound}. We observe that $Q_0-u \in \mathcal{H}^2_{(0)}$, as $Q_0, u\in \mathcal{H}^2$ and they have the same boundary value. For every $\epsilon>0$, by Theorem \ref{thm:densityofimB}, there exists a choice of $v^{(\epsilon)}_0\in L^2_\eta$ such that 
$\|Q_0 - u - \mathcal{B}v^{(\epsilon)}_0\|_{\mathcal{H}^2}\leq \epsilon$. We further recall that $\|v_0^{(\epsilon)}\|^2_{\mathcal{B}}<\infty$ for $v_0^{(\epsilon)}\in L^2_\eta$. 

Defining $g^{(\epsilon)}=Q_0 - u - \mathcal{B}v^{(\epsilon)}_0$, \eqref{Qsquarebound} yields the limit as $t\to \infty$
\begin{align}\frac{1}{t}\int_0^t\|Q_s-u\|^2_{L^2} ds \leq C\|g^{(\epsilon)}\|^2_{\mathcal{H}^2} + \frac{C}{t}\|v_0^{(\epsilon)}\|^2_{\mathcal{B}}\to C\epsilon.\end{align}
As $\epsilon>0$ was arbitrary, we conclude that
$
    \lim_{t\to \infty}\big\{\frac{1}{t}\int_0^t\|Q_s-u\|^2_{L^2} ds\big\} = 0$.

To obtain the final statement, set
$\beta(t) = \sup_{T\ge t} \big\{\frac{1}{T}\int_0^T\|Q_s-u\|^2_{L^2} ds\big\}$.
Observe that $\beta$ is a nonincreasing positive function with $\lim_{t\to\infty}\beta(t)= 0$.
Defining the set $A = \big\{t:\|Q_t-u\|^2> \sqrt{\beta(t)}\big\}$,
Markov's inequality yields
\begin{align}
\begin{split}
    \frac{1}{t}\mathrm{Leb}(A\cap[0,t])
    &=\frac{1}{t}\mathrm{Leb}\Big(\Big\{s\leq t: \frac{\|Q_s-u\|^2_{L^2}}{\sqrt{\beta(s)}}>1\Big\}\Big) \leq \frac{1}{t}\int_0^t\frac{\|Q_s-u\|^2_{L^2}}{\sqrt{\beta(s)}}ds\\
    &\leq \frac{1}{\sqrt{\beta(t)}}\Big(\frac{1}{t}\int_0^t\|Q_s-u\|^2_{L^2}\,ds\Big) \leq \sqrt{\beta(t)} \to 0.
    \end{split}
\end{align}
\end{proof}
The above result only considers the convergence of $Q$. In the case where $\mathcal{L}$ is Gateaux differentiable (which is certainly the case when $\mathcal{L}$ is linear), we can give a similar result for the convergence of $\mathcal{L}Q_t$, involving the positive-definite kernel $\mathcal{B}$.

\begin{theorem}\label{thm:diffLconvergence}
Suppose $\mathcal{L}:\mathcal{H}^2\to L^2$ is Gateaux differentiable and the conditions of Theorem \ref{thm:Qconvergence} hold. Then there exists $K>0$ such that 
$\int_0^t \|\mathcal{B}\mathcal{L}Q_s\|^2_{L^2}\, ds \leq K$.
In particular, there exists a set $A\subset[0,\infty)$ with $\mathrm{Leb}(A) <\infty$ such that for all sequences $t_k\to\infty$ which do not take values in $A$, we have 
$\|\mathcal{B}\mathcal{L}Q_{t_k}\|_{L^2}\to 0$.
\end{theorem}
\begin{proof}
    For any $w\in L^2$, $h\in \mathcal{H}^2$ we write $\partial\mathcal{L}(w; h) = \lim_{\epsilon\to 0}\{(\mathcal{L}(w+h) - \mathcal{L}(w))/\epsilon\}$ for the Gateaux derivative in direction $h$ evaluated at $w$, the limit being taken in $L^2$. 
    
    We first observe that, as $\mathcal{L}$ is strongly monotone, for any $h\in \mathcal{H}^2$,
    \begin{equation}\langle h, \partial\mathcal{L}(w;h)\rangle = \lim_{\epsilon\to 0} \frac{1}{\epsilon}\langle h, \mathcal{L}(w+\epsilon h)- \mathcal{L}(w)\rangle
    \leq -\gamma\lim_{\epsilon\to 0} \frac{1}{\epsilon^2}\|\epsilon h\|^2_{L^2} = -\gamma \|h\|^2_{L^2}.\end{equation}
    Now write $v_t = \mathcal{L}Q_t$. As $Q_t\in \mathcal{H}^2$, we know that $\mathcal{B}v_t = \mathcal{B}\mathcal{L}Q_t \in \mathcal{H}^2$ (as $\mathcal{L}:\mathcal{H}^2\to L^2$ and $\mathcal{B}:L^2\to \mathcal{H}^2$). By the chain rule, we have ${d v_t}/{dt} = \partial\mathcal{L}\big(Q_t;{dQ_t}/{dt}\big) = \partial\mathcal{L}\big(Q_t;\mathcal{B}\mathcal{L}Q_t\big)= \partial\mathcal{L}\big(Q_t;\mathcal{B}v_t\big)$.    Therefore, 
    \begin{equation}
        \frac{d}{dt}\|v_t\|^2_{\mathcal{B}} = 2\Big\langle \mathcal{B}v_t, \frac{dv_t}{dt}\Big\rangle = 2\Big\langle \mathcal{B} v_t, \partial\mathcal{L}\big(Q_t;{\mathcal{B}v_t}\big)\Big\rangle \leq -2\gamma \|\mathcal{B} v_t\|^2_{L^2}.
    \end{equation}
    Lemma \ref{lem:BLip} gives the bound $\|\mathcal{B}v_t\|^2_{L^2}\leq \lambda\|v_t\|_{\mathcal{B}}$, from which we obtain
    \begin{equation}\|\mathcal{B} v_t\|^2_{L^2} \leq \lambda\|v_t\|^2_{\mathcal{B}} \leq \lambda\|v_0\|^2_{\mathcal{B}}-2\lambda\gamma\int_{0}^t\|\mathcal{B}v_s\|^2 ds.\end{equation}
    By Gr\"onwall's inequality, we conclude
    \begin{equation}\int_0^t\|\mathcal{B}v_s\|^2_{L^2}ds \leq \lambda\|v_0\|^2_{\mathcal{B}}\int_0^t e^{-2\lambda\gamma (t-s)}ds \leq \frac{\|v_0\|^2_{\mathcal{B}}}{2\gamma } =: K<\infty.\end{equation}
Taking $A = \{t:\|\mathcal{B}v_t\|^2_{L^2}> 1/t\}$, the final stated property is obtained from Markov's inequality, similarly to in Theorem \ref{thm:Qconvergence}.
\end{proof}
\begin{corollary}
Under the conditions and notation of Theorem \ref{thm:diffLconvergence}, the sequence $\mathcal{L}Q_{t_k}$ converges weakly to zero in $L^2$ if and only if it remains bounded in $L^2$.
\end{corollary}
\begin{proof}
    Recall that a sequence $x_n$ converges  $L^2$-weakly to zero if and only if it is  $L^2$-bounded and there exists an $L^2$-dense set $\mathcal{Y}$ such that $\langle y, x_n\rangle \to 0$ for all $y\in \mathcal{Y}$.
    
    We know from Theorem \ref{thm:diffLconvergence} that $\mathcal{B}\mathcal{L}Q_{t_k}\to 0$ strongly in $L^2$ so, for any $v\in L^2$,
$        \langle \mathcal{B}v, \mathcal{L}Q_{t_k}\rangle = \langle v, \mathcal{B}\mathcal{L}Q_{t_k}\rangle \to 0$.
    In particular, we see that $\langle y, Q_{t_k}\rangle \to 0$ for any $y\in \mathrm{im}(\mathcal{B})$. We know from Theorem \ref{thm:densityofimB} that $\mathrm{im}(\mathcal{B})$ is dense in $\mathcal{H}^2$, and hence in $L^2$. The result follows.
\end{proof}

\section{Numerical experiments}\label{sec_num}
In this section, we present numerical results where we apply our algorithm to solve a family of partial differential equations. The approximator matches the solution of the differential equation closely in a relatively short period of training time in these test cases.
\subsection{Discretization of the continuous-time algorithm} \label{sec_61}
As a reminder, in a continuous time setting, the algorithm follows a biased gradient flow:
${d \theta_t}/{dt}=-\alpha^N_t G_t^N(\theta_t)$,
based on the biased gradient estimator
\begin{align}
    G_t^N(\theta_t)=\int_{\Omega} 2\psi(\mathcal{L}Q^N_t(x))\psi'(\mathcal{L}Q^N_t(x))\nabla_\theta (-\gamma Q_t^N(x))d\mu(x).
\end{align}
A natural discretization of our algorithm with which we train our approximator is
\begin{align} \label{177}
    \theta_{t+1}-\theta_t = -{\alpha}_t^N \hat {G_{M,t}^N},
\end{align}
where \begin{align}\label{GN}
    \hat {G_{M,t}^N} = \frac{1}{M}\sum_{i=1}^{M} 2\psi(\mathcal{L}Q^N_t(x_t^i))\psi'(\mathcal{L}Q^N_t(x_t^i))\nabla_\theta (-\gamma Q_t^N(x_t^i))
\end{align}
is an unbiased estimator of $G_t^N$, given $\{x_t^i\}_{i=1,2,..,M}$ the random grid independently sampled from distribution $\mu$ at time step $t$, and ${\alpha}_t^N$ is the learning rate. More advanced discretization methods, for example the ADAM adaptive gradient descent rule (\cite{kingma2014adam}) can also be used, and (as is common in gradient based optimization problems) give a noticeable improvement in performance.

The Q-learning algorithm for solving PDEs as follows:

\smallskip
\begin{algorithm}[H] \label{algo}
\SetAlgoLined
 \textbf{Parameters:} Hyper-parameters of the single-layer neural network; Domain $\Omega$; PDE operator $\mathcal{L}$; Boundary condition at $\partial \Omega$; Sampling measure $\mu$; Number of Monte-Carlo points $M$\; Upper bound of training time $T$\;
 \textbf{Initialise:} Neural net $S^N$; Auxiliary function $\eta$; Approximator $Q^N$ based on $S^N$ and $\eta$; Smoothing function $\psi^N$; Learning rate scheduler $\{{\alpha}_t^N \}_{t\geq 0}$; Stopping criteria $\epsilon$; Current time $t=0$.

 \While{$err \geq \epsilon$ and $t\leq T$}{
  Sample $M$ points in $\Omega$ using $\mu$, $\{x_i\}$\;
  Compute biased gradient estimator ${G_{M,t}^N}$ using \eqref{GN}\;
  Update neural network parameters via (\ref{177})\;
  Compute $err=\frac{1}{M}\sum_{i=1}^{M}\psi^N(\mathcal{L}Q_t^N(x_i))^2$\;
  Update time $t$;
 }
 Return approximator $Q_t^N$.
\caption{Q-PDE Algorithm}
\end{algorithm}

\smallskip

We implement this algorithm using PyTorch (with backpropagation to compute gradients in both parameters and the state variable $x$) and the standard ADAM optimization scheduler for $\alpha$. \footnote{The implementation is available at \url{https://github.com/DeqingJ/QPDE}.}

\subsection{Test equation: Survival time of a Brownian motion}
Suppose $\Omega$ is the $n$-dimensional unit ball. For $n$-dimensional Brownian motion $X_t$ starting from point $x \in \mathbb{R}^n$, we wish to compute
\begin{align}
    u(x):=\mathbb{E}\Big[\int_0^{\tau_\Omega}e^{-\gamma t} dt\Big|X_0=x\Big ]
\end{align} where $\tau_\Omega$ is the first exit time of $X$ for domain $\Omega$. By the Feynman--Kac theorem (applied to the process $X$ stopped at $\tau_\Omega$), this expectation is given by the solution of the following PDE:
\begin{align}
\begin{cases} \label{examplepde1}
    1 - \gamma u + \frac{1}{2}\Delta u&= 0\quad \text{in} \,\,  \Omega\\
    \quad \quad \quad \quad \quad \,\,\,\,u&= 0 \quad \text{on}\,\, \partial \Omega,
\end{cases}
\end{align}
where $\Delta u$ is the Laplacian of $u$. We consider this problem as a good test case as its explicit solution is known to us, allowing us to check if our algorithm works for a high-dimensional version of this PDE.
We initialize a single neural network $S^N$ equipped with sigmoid activation function, and introduce auxiliary function $\eta(x)=1-\|x\|^2$. The approximator $Q=S\cdot \eta$ is trained to fit the solution. As shown in the following subsections, through our algorithm, the approximator $Q$ learns to have the solution of the PDE with different discounting factor in different dimensions. For detailed configuration of our approximator, see Appendix \ref{config}.

\subsubsection{1-dimensional case}
For dimension $n=1$, domain $\Omega=(-1,1)$. The exact solution of this differential equation is $u(x)=\frac{1}{\gamma}+c_1 (e^{\sqrt{2\gamma} x} + e^{-\sqrt{2\gamma} x})$
where $c_1={-1}/{(\gamma (e^{-\sqrt{2\gamma}} +e^{\sqrt{2\gamma}}))}$. For this subsection, set $\gamma=0.1$. To keep track of the training progress, we monitor the average loss level at time $t$, that is $e_t \approx \int_\Omega [\mathcal{L}Q_t]^2 d\mu$ where $\mu$ is the Lebesgue measure on $\Omega$, which can be estimated using our sample of evaluation points at each time.

\begin{figure}[H]
            \centering
            \includegraphics[width=\textwidth]{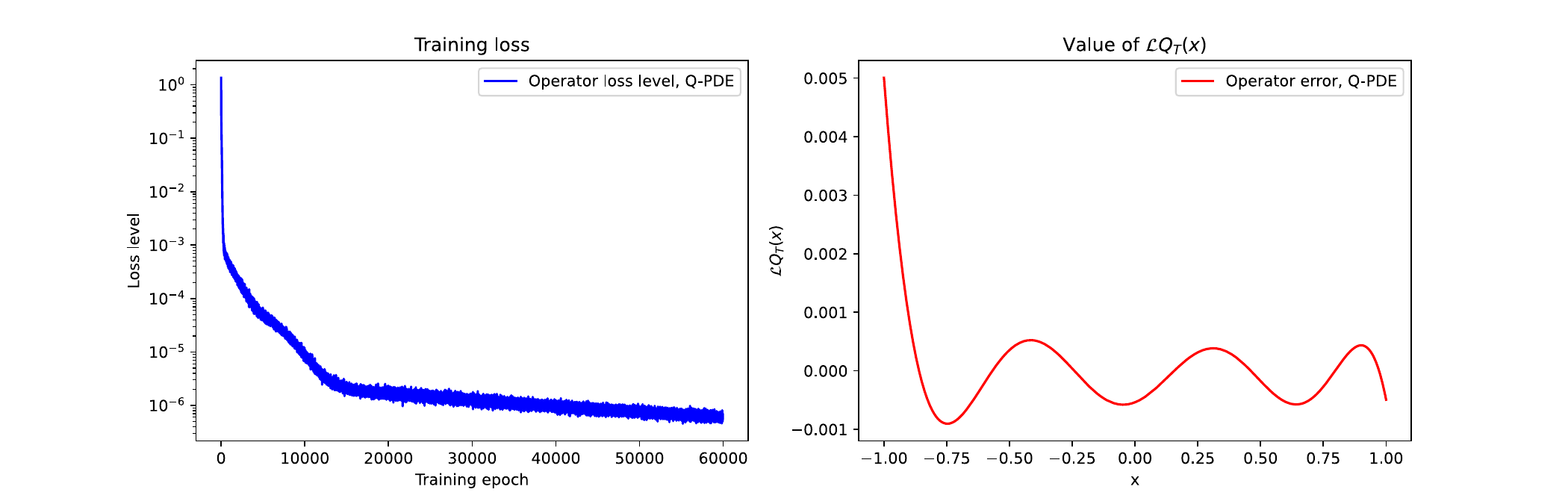}
            \caption{Loss level during and after training}
            \label{f1}
        \end{figure}
The left side of Figure \ref{f1} shows how the average operator loss smoothly decays during training. The right side of it plots the operator error at terminal time $\mathcal{L}Q_T^N(x)$. For a perfect fit, the operator loss $\mathcal{L}Q_T(x)$ should be zero across the entire domain. We observe that for our approximator, the norm of this loss remains small and fluctuates around $0$. In Figure \ref{f2}, we compare our approximator $Q_\tau^N$ with the exact solution. Here $\tau$ indicates the training time at which we observed the minimal loss $e_\tau$, which occurs for $\tau$ close to the terminal time $T$. (Given the fluctuations observed in the loss, chosing this minimal point can give a noticeable qualitative improvement in performance.) The relative error of our approximation at $x$, given by $|{(Q_\tau^N(x)-u(x))}/{u(x)}|$, remains lower than $0.02\%$.
\begin{figure}[H]
            \centering
            \includegraphics[width=\textwidth]{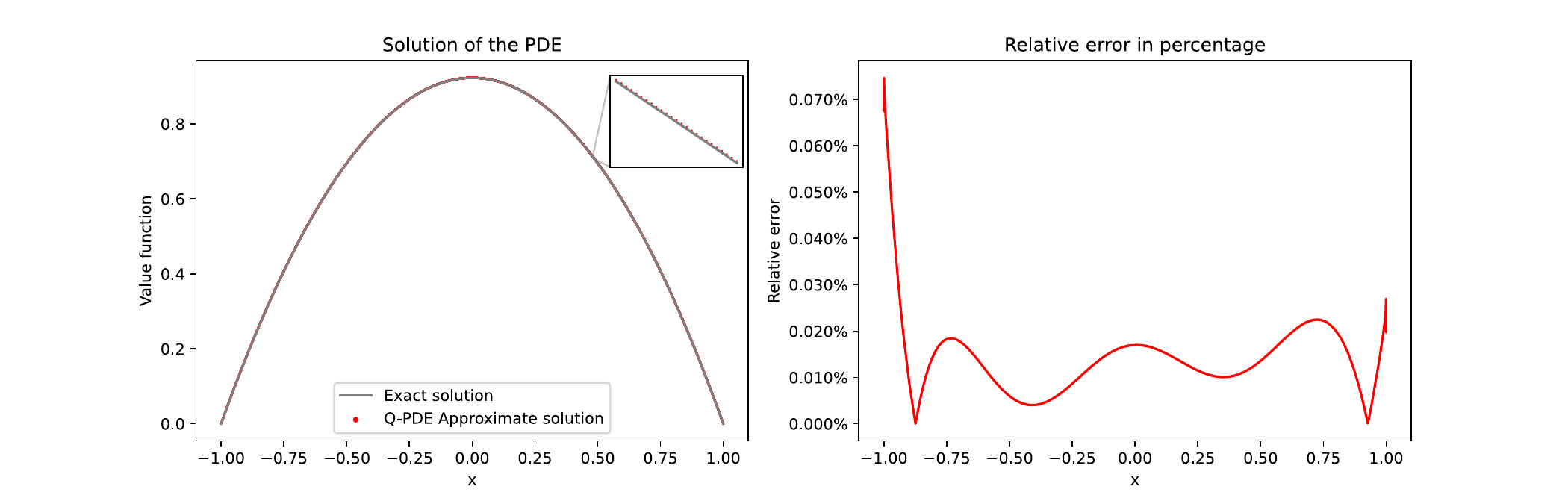}
            \caption{Approximate and exact solutions in 1d case}
            \label{f2}
        \end{figure}
        
\subsubsection{20-dimensional case}
Now we consider the 20-dimensional case. High-dimensional PDEs are typically hard to solve, as mesh grids are computationally unaffordable in these cases. However, as a mesh-free method, our algorithm can be applied. We numerically solve the PDE  \eqref{examplepde1} for discounting factor $\gamma=0.2$ via our algorithm and compare the result with the exact solution, given by
\begin{equation}
    u(x)= \frac{1}{\gamma}\Big(1 - \frac{J_9\big(i \sqrt{2\gamma}\|x\|\big)}{\|x\|^9 J_9(i \sqrt{2\gamma})}\Big),
\end{equation}
where $J_9$ is the ninth-order Bessel function of the first kind. We also compare the training results of our method to the results from the Deep Galerkin Method (which does stochastic gradient descent with an unbiased gradient estimate) as a benchmark.

Our approximator trained with the Q-PDE algorithm is able to accurately approximate the solution of the 20-dimensional PDE. Left of Figure \ref{f5} shows that the operator loss is less than $10^{-5}$ after training for 40,000 epochs, for both Q-PDE and DGM. For both methods, we observe cyclic behavior in the loss during training, which appears similar to the slingshot effect discussed in \cite{thilak2022slingshot}, which is potentially a consequence of momentum accumulation of the ADAM optimizer. A full understanding of this effect (which we have not observed when using a simple gradient descent optimizer) is left for future work.
\begin{figure}[H]
            \centering
            \includegraphics[width=\textwidth]{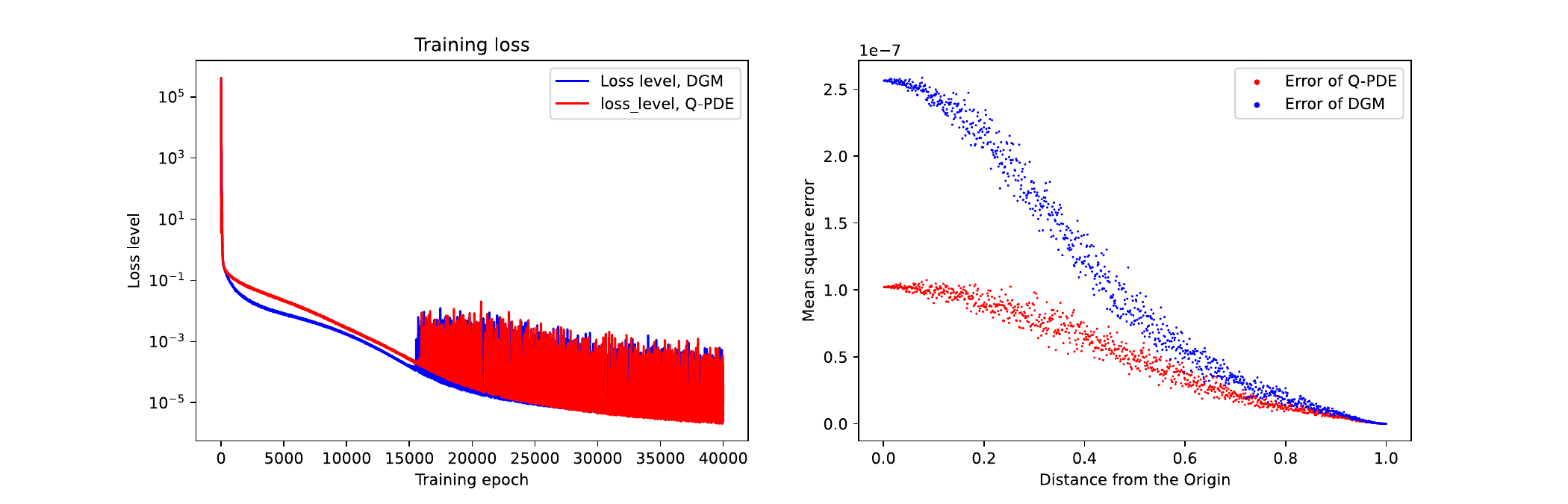}
            \caption{Average loss level during training}
            \label{f5}
        \end{figure}
The solution $u$ of this PDE is radially symmetric, that is, it can be represented by a function which depends on $\|x\|$ rather than $x$. For the right side of Figure \ref{f5}, we verify that our approximate solution is close to this function for different $\|x\| \in [0,1]$. For $0\leq r \leq 1$, we randomly generate 100 sample points $\{x_{j,r}\}_{j=1,...,100}$ on the sphere $\|x\| = r$ and compute the mean-squared error of the set of values $\{Q_\tau^N(x_{i,r})\}$ as compared to the exact solution. It is shown that for different $r$, the mean-squared error of the neural network approximator is consistently small. Q-PDE solution obtains smaller error than our DGM solution, but they are at the same scale.

For $0\leq r \leq 1$, we also consider the maximum relative error, defined as 
\begin{equation}
    R_r = \max_{x_{i,r}} \bigg\{\Big|\frac{Q^N_\tau(x_{i,r})-u(x_{i,r})}{u(x_{i,r})}\Big|\bigg\}.
\end{equation}
\begin{figure}[H]
            \centering
            \includegraphics[width=\textwidth]{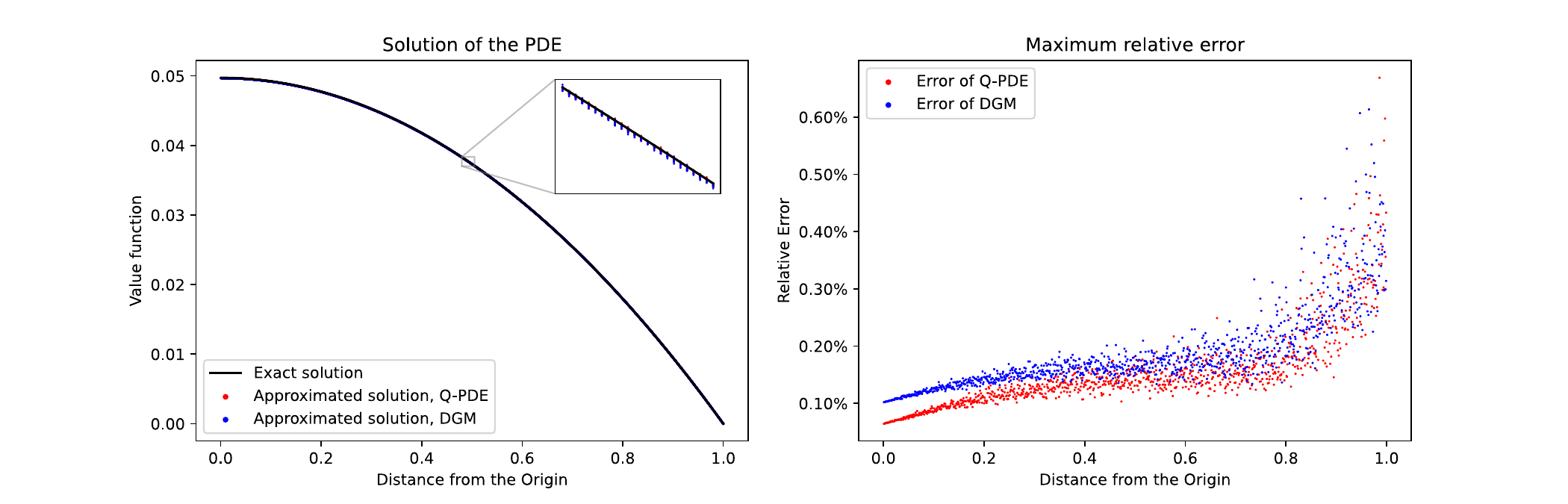}
            \caption{Relative and absolute errors of the approximate solution, in terms of distance from the origin.}
            \label{f6}
        \end{figure}
The left plot of Figure \ref{f6} implies that this error remains small for all $0 \leq r \leq 1$. Again, DGM solution admits slightly larger relative error. Meanwhile, as DGM solution is trained with exact gradients which are computationally more expensive, consequently, it takes 5 times longer to train. Both DGM and Q-PDE solutions reach high accuracy.
We plot a scatter plot for $\{Q_\tau^N(w_i)\}$ in the right side in Figure \ref{f6} and compare with the exact solution $u$. The approximate solutions closely matches the exact solution. 

\section{Conclusion}
 In this paper, we propose a novel algorithm which numerically solves elliptic PDEs. The neural network approximator is designed to match the Dirichlet boundary condition of the PDE and is trained to satisfy the PDE operator. We prove that for single-layer neural network approximators, as the number of hidden units $\rightarrow \infty$, the evolution of the approximator during training is characterized by a limit ODE. We further prove that as the training time $\rightarrow \infty$, the solution of the limit ODE converges to the classical solution of the PDE. In addition, we provide a numerical test case to show that our algorithm can numerically solve (high-dimensional) PDEs.
\acks{Thanks to Endre S\"uli, {\L}ukasz Szpruch, Christoph Reisinger and Rama Cont, and to anonymous referees, for useful conversations and suggestions.

This publication is based on work supported by the EPSRC Centre for Doctoral Training in Industrially Focused Mathematical Modelling (InFoMM) (EP/L015803/1) in collaboration with the Alan Turing Institute. Samuel Cohen also acknowledges the support of the Oxford--Man Institute for Quantitative Finance, and the UKRI Prosperity Partnership Scheme (FAIR) under the EPSRC Grant EP/V056883/1.

The order of authorship in this paper was determined alphabetically.}


\newpage

\appendix
\section{Appendix}
In our proofs, the constants $K$, $C$ may vary from line to line.
\subsection{Proof of Lemma \ref{lemma_eta}} \label{proof_214}
\begin{proof}
We first prove\footnote{Thanks to Endre S\"uli for suggesting the proof of statement \emph{i} given here.} statement \emph{i}. Suppose that $u \in \mathcal{H}^2(\Omega) \cap \mathcal{H}^1_0(\Omega)$. Consider the Poisson problem solved by $u$:
\begin{align}
\begin{cases} 
    -\Delta u= f\quad \text{in} \,\,  \Omega,\\
    \quad \,\,\,\,u= 0 \quad \text{on}\,\, \partial \Omega,
\end{cases}
\end{align}
where $f:= - \Delta u \in L^2(\Omega)$. Let $f_n \in C^\infty_0(\Omega)$ be a sequence with $f_n \to f \in L^2(\Omega)$. Then consider the Poisson problem
\begin{align}
\begin{cases} 
    -\Delta u_n= f_n\quad \text{in} \,\,  \Omega,\\
    \quad \,\,\,\,u_n= 0 \quad \,\,\text{on}\,\, \partial \Omega.
\end{cases}
\end{align}
Since (by Assumption \ref{boundaryRegular}) $\partial \Omega \in C^{3,\alpha}$ for a H\"older
exponent $\alpha \in (0,1)$ and $f_n\in C^\infty_0(\Omega)\subset C^{1,\alpha}(\overline \Omega)$, by the Schauder theory for elliptic
boundary-value problems (see, for example, Theorems 6.14 and 6.19 in \cite{gilbarg1998elliptic}), there exists a unique solution $u_n \in C^{3,\alpha}(\overline\Omega) \cap C_0(\overline\Omega)$ to this boundary-value problem. As
\begin{align}
\begin{cases} 
    -\Delta (u-u_n)= (f-f_n)\quad \text{in} \,\,  \Omega,\\
    \quad \quad \,\,u-u_n= 0 \quad \quad \quad \quad \,\text{on}\,\, \partial \Omega,
\end{cases}
\end{align}
and $\partial \Omega \in C^{3,\alpha}$ (and
therefore also $C^2$), and $u - u_n \in \mathcal{H}^2(\Omega)$, it follows
from Theorem 3.1.2.1 and Remark 3.1.2.2 in \cite{grisvard2011elliptic}, that for some $C>0$ \begin{align}
    \|u - u_n\|_{\mathcal{H}^2} \leq C\|\Delta (u - u_n)\|_{L^2} = C \|f - f_n\|_{L^2} \to  0
\end{align}
as $n \to \infty$. Thus, we have constructed a sequence of functions $u_n \in C^{3,\alpha}(\overline\Omega)\cap \mathcal{H}^1_0$ which converges to $u \in \mathcal{H}^2(\Omega)$ in the $\mathcal{H}^2$ norm. Therefore, $C^{3,\alpha}(\overline\Omega) \cap C_0(\overline\Omega)$ is dense in $\mathcal{H}^2(\Omega)\cap \mathcal{H}^1_0(\Omega)$. As
\begin{align}
    \Big(C^{3,\alpha}(\overline\Omega) \cap C_0(\overline\Omega)\Big) \subset \Big(C^{3}(\overline\Omega) \cap  C_0(\overline\Omega)\Big),
\end{align}
we conclude that $C^{3}(\overline\Omega) \cap  C_0(\overline\Omega)$ is also dense in $\mathcal{H}^2(\Omega) \cap \mathcal{H}^1_0(\Omega)$.

We now prove statement \emph{ii}. Take an arbitrary $u\in C^3(\overline\Omega)\cap  C_0(\overline\Omega)$, and define $\tilde u = u/\eta$. It is clear that $\tilde u|_\Omega$ is in $C^3(\Omega)$, from classical rules of calculus, as $\eta|_\Omega> 0$. In particular, $\tilde u$ and its derivatives are bounded on any closed subset of $\Omega$, so we only need to consider the behavior of $\tilde u$ in a neighbourhood of $\partial\Omega$.

We know that $\eta$ has nonvanishing derivative on $\partial\Omega$, and equals zero on $\partial \Omega$, which implies that $|\langle \nabla \eta(x), \mathbf{n}_{x}\rangle|>0$ for each $x\in\partial\Omega$, where $\mathbf{n}_x$ is the outward pointing normal at $x$.  As  $\partial \Omega$ is compact, this implies that
\begin{equation}\inf_{x \in \partial\Omega}|\langle \nabla \eta(x), \mathbf{n}_{x}\rangle|>\epsilon,\quad \text{for some }\epsilon>0.\end{equation} 
As $u\in C^3(\overline\Omega)\cap  C_0(\overline\Omega)$, for any $x\in \partial \Omega$, we also have $\mathop{\mathrm{lim\,sup}}_{x'\to x}|u(x')| = 0$.

Consider a smooth path $\rho:[0,1)\to \Omega$ with $\lim_{t\to1}\rho(t)=x\in \partial \Omega$. For a function $g$ with domain $\Omega$ (e.g. $g=\eta$, $u$ or $\tilde u$), we write $g_\rho$ for the composition $g\circ \rho$ with domain $[0,1)$. Suppose $\rho$ does not approach $\partial \Omega$ tangentially, in particular, 
\begin{equation}\label{nontangent}
\lim_{t\to 1}\langle d\rho/dt, \mathbf{n}_{x}\rangle \geq \epsilon \quad \text{and}\quad \|d\rho/dt\|\equiv 1.
\end{equation}
As $\langle \nabla \eta(x), y\rangle=0$ for all $y\perp \mathbf{n}_{x}$, we have
\begin{equation}
    \lim_{t\to 1}|\eta_\rho'(t)|= |\langle \nabla \eta(x), \mathbf{n}_{x}\rangle|\, \langle d\rho/dt, \mathbf{n}_{x}\rangle>\epsilon^2.
\end{equation}

Using L'H\^opital's rule, we can determine the behaviour of $\tilde u_\rho$ (and its derivatives) as $t\to 1$. In particular, by repeated application of the rule,  as $t\to 1$
\begin{align}
    \tilde u_\rho &= \frac{u_\rho}{\eta_\rho} \to \frac{u_\rho'}{\eta_\rho'}\\
    \tilde u_\rho' & = \frac{\eta_\rho u_\rho' - \eta_\rho' u_\rho}{\eta_\rho^2} \to  \frac{u_\rho''}{ 2\eta_\rho'} -  \frac{\eta_\rho''}{2\eta_\rho'}\tilde u_\rho\\
    \tilde u_\rho'' &= \frac{\eta_\rho^2 u_\rho'' - \eta_\rho \eta_\rho'' u_\rho - 2\eta_\rho \eta_\rho' u_\rho' + 2 (\eta_\rho')^2u_\rho}{\eta_\rho^3}\to \frac{u_\rho'''}{3\eta_\rho'} - \frac{\eta_\rho'''}{3\eta_\rho'}\tilde u_\rho+ \frac{\eta_\rho''}{\eta_\rho'}\tilde u_\rho'
\end{align}
Considering $t\approx 1$, as $|\eta'_\rho|>\epsilon^2$, the right hand side of each of these terms is bounded,  with a uniform bound for all paths under consideration. Writing $D_h(\tilde u)$ for the directional derivative of $\tilde u$ in a direction $h$, we know $\|\tilde u'_\rho\| = \|D_{d\rho/dt}(\tilde u)\|$ and $\|\tilde u_\rho''\| = \|D^2_{d\rho/dt}(\tilde u)\|$. We can therefore find a single value $K>0$, and take a ball $B_x$ around each $x\in \partial\Omega$, such that, within each $B_x$,  for all unit vectors $h$ satisfying $\langle h, \mathbf{n}_x\rangle >\epsilon$,
\begin{equation}\label{tildeubound}
|\tilde u|+|D_{h}(\tilde u)|+|D_{h}^2(\tilde u)|\leq K.
\end{equation}
 As $\Omega\setminus (\cup_{x\in\partial \Omega} B_{x})$ is compact, we know that, for $K$ sufficiently large, \eqref{tildeubound} also holds on $\Omega\setminus (\cup_{x\in\partial \Omega} B_{x})$, for all unit vectors $h$. We conclude that \eqref{tildeubound} holds on all of $\Omega$, that is, $\tilde u$, and its directional first and second derivatives (except possibly in tangent directions at the boundary), are uniformly bounded.

However, as $\partial\Omega$ is a $C^{3,\alpha}$ boundary (in particular, for all $\epsilon$ sufficiently small and any $x\in \partial\Omega$ the set $\{y\in \mathbb{R}^n: \langle y, \mathbf{n}_{x}\rangle \geq \epsilon\}$ is a convex cone containing a linear basis for $\mathbb{R}^n$), this implies that the first and second derivatives of $\tilde u$ are also uniformly bounded on $\Omega$. We conclude that $\tilde u\in C^2_b(\Omega)$.
\end{proof}
\subsection{Proof of Lemma \ref{bounded_t}}\label{proof_2}
\begin{proof}
(1) From the training algorithm, calculating the entries in vector $\nabla_\theta Q_t^N(x)$, we have 
\begin{align}
\begin{split} \label{training_dynamic}
    \frac{\partial Q_t^N}{\partial c^i}(x)&=\frac{\eta(x)\sigma( w_t^i \cdot x+b_t^i)}{N^\beta}. 
\end{split}
\end{align}
Therefore, by (\ref{Q-Learning}), applying the chain rule we derive
\begin{align} \label{dynamic_c}
    \frac{dc_t^i}{dt}=\alpha N^{\beta-1} \int_{\Omega}F^N(\mathcal{L}Q^N_t(x))\eta(x)\sigma(w_t^i \cdot x+b_t^i) d\mu(x).
\end{align}
By the boundedness of the functions $F^N$, $\eta$ and $\sigma$, the RHS of (\ref{dynamic_c}) is bounded:
\begin{align} \label{dct}
        \bigg |\int_{\Omega}F^N(\mathcal{L}Q^N_t(x))\eta(x)\sigma(w_t^i \cdot x+b_t^i) d\mu(x) \bigg |< KN^{\delta},
\end{align}
    where $K$ is a constant. Therefore, we conclude that
    \begin{align}
        |c_t^i|<|c_0^i|+|c_t^i-c_0^i|<K_0+\alpha N^{\delta+\beta-1} tK < C.
    \end{align}

(2) Similar to above, considering the $w^i$ entries in the vector $\nabla_\theta Q_t^N(x)$, we have 
\begin{align}
\begin{split}
    \frac{\partial Q_t^N}{\partial (w_t^i)_k}(x)&= \frac{\eta(x)\sigma'(w_t^i \cdot x +b_t^i)c_t^i x_k}{N^\beta}. 
\end{split}
\end{align}
Therefore $w_t^i$ satisfies:
\begin{align} \label{dynamic_w}
    \frac{d(w_t^i)_k}{dt}=\alpha N^{\beta-1} \int_{\Omega}F^N(\mathcal{L}Q^N_t(x))\eta(x)c_t^i\sigma'(w_t^i \cdot x+b_t^i)x_k d\mu(x).
\end{align}
By the boundedness of $F^N$, $\eta$, $\sigma'$ and $c_t^i$, the RHS of (\ref{dynamic_w}) is bounded by
\begin{align} \label{dwt}
    \bigg |\int_{\Omega}F^N(\mathcal{L}Q^N_t)(x)\eta(x)c_t^i\sigma'(w_t^i \cdot x+b_t^i)x_k d\mu(x) \bigg |< KN^{\delta},
\end{align}where $K$ is a constant. Therefore, \begin{align}
    |(w_t^i)_k|<|(w_0^i)_k|+|(w_t^i)_k-(w_0^i)_k|<|(w_0^i)_k|+\alpha N^{\delta+\beta-1} tK<|(w_0^i)_k|+C.
\end{align}
Taking an expectation on both sides, we obtain \begin{align}
    \mathbb{E}|(w_t^i)_k|<C.
\end{align}

(3) Considering the $b^i$ entries in the vector $\nabla_\theta Q_t^N(x)$, we have 
\begin{align}
\begin{split}
    \frac{\partial Q_t^N}{\partial b_t^i}(x)&= \frac{\eta(x)\sigma'(w_t^i \cdot x +b_t^i)c_t^i }{N^\beta}. 
\end{split}
\end{align}
Therefore, $b_t^i$ satisfies
\begin{align} \label{dynamic_b}
    \frac{d b_t^i}{dt}=\alpha N^{\beta-1} \int_{\Omega}F^N(\mathcal{L}Q^N_t(x))\eta(x)c_t^i\sigma'(w_t^i \cdot x+b_t^i) d\mu(x).
\end{align}
By the boundedness of $F^N$, $\eta$, $\sigma'$ and $c_t^i$, the RHS of (\ref{dynamic_w}) is bounded by
\begin{align} \label{dbt}
    \bigg |\int_{\Omega}F^N(\mathcal{L}Q^N_t)(x)\eta(x)c_t^i\sigma'(w_t^i \cdot x+b_t^i)d\mu(x) \bigg |< KN^{\delta},
\end{align}where $K$ is a constant. Therefore, \begin{align}
    |b_t^i|<|b_0^i|+|b_t^i-b_0^i|<|b_0^i|+\alpha N^{\delta+\beta-1} tK<|b_0^i|+C.
\end{align}Taking an expectation on both sides, we obtain \begin{align}
    \mathbb{E}|b_t^i|<C.
\end{align}
\end{proof}
\label{appen}
\subsection{Proof of Lemma \ref{ata0_lemma_1}}
\begin{proof} (1) Notice that \begin{align}
    \begin{split} \label{ata0_1}
        &|A_t^N(x,y)-A_0^N(x,y)| =\Bigg | \frac{1}{N}\sum_{i=1}^N \bigg [\sigma(w_t^i \cdot x+b_t^i)\sigma(w_t^i \cdot y + b_t^i)-\sigma(w_0^i \cdot x+b_0^i)\sigma(w_0^i \cdot y + b_0^i)\\
        & \qquad + x \cdot y (C_t^i)^2 \sigma'(w_t^i \cdot x+b_t^i)\sigma'(w_t^i \cdot y + b_t^i) -x \cdot y (C_0^i)^2 \sigma'(w_0^i \cdot x+b_0^i)\sigma'(w_0^i \cdot y + b_0^i) \bigg ] \Bigg |\\
        &\leq  \frac{1}{N}\sum_{i=1}^N \bigg [ |\sigma(w_t^i \cdot x+b_t^i)\sigma(w_t^i \cdot y + b_t^i)-\sigma(w_0^i \cdot x+b_0^i)\sigma(w_0^i \cdot y + b_0^i)|\\
        &\qquad +|x \cdot y\|{c_t^i}^2-{c_0^i}^2\|\sigma'(w_t^i \cdot x+b_t^i)\sigma'(w_t^i \cdot y + b_t^i)|\\
        &\qquad + {c_0^i}^2|x \cdot y\|\sigma'(w_t^i \cdot x+b_t^i)\sigma'(w_t^i \cdot y + b_t^i)-\sigma'(w_0^i \cdot x+b_0^i)\sigma'(w_0^i \cdot y + b_0^i)|\bigg ].
    \end{split}
    \end{align}
    For a smooth function $f:\mathbb{R} \to \mathbb{R}$, applying the mean value theorem with respect to $w_t^i$, there exists $w^* \in \mathbb{R}^n$ such that
    \begin{align}
        \begin{split}
            &f(w_t^i \cdot x+b)f(w_t^i \cdot y + b)-f(w_0^i \cdot x+b)f(w_0^i \cdot y + b)\\
            &=(w_t^i-w_0^i)\cdot\Big[f'(w^* \cdot x+b))f(w^* \cdot y+b)x+f(w^* \cdot x+b))f'(w^* \cdot y+b)y\Big].
        \end{split}
    \end{align}
    By the fact that $\sigma \in C_b^4(\mathbb{R})$ and $|c_u^i|< C$ for any $u \in [0,T]$, we have 
    \begin{align}
    \begin{split}
         &|A_t^N(x,y)-A_0^N(x,y)| \leq  \frac{C}{N}\sum_{i=1}^N \Big[(\|x\|_1+\|y\|_1)\|w_t^i-w_0^i\|_1+|x \cdot y| |c^i_t-c_0^i|\\
         &\qquad +\|w_t^i-w_0^i\|_1(\|x\|_1+\|y\|_1)|x \cdot y|+(1+|x \cdot y|)|b_t^i -b_0^i|\Big].
    \end{split}
    \end{align}
    Therefore, by \eqref{dct}, \eqref{dbt} and \eqref{dwt}, the increments of $c_t$, $b_t$ and $w_t$ are bounded by $CN^{\delta+\beta-1}$. Therefore, we have 
    \begin{align}
    \begin{split}
        |A_t^N(x,y)-A_0^N(x,y)|&\leq CN^{\delta+\beta-1}\bigg [ (\|x\|_1+\|y\|_1)(1+|x\cdot y|)+ |x \cdot y| \bigg ] \leq CN^{\delta+\beta-1}.
    \end{split}
    \end{align}
    (2) By definition,
    \begin{align}
        \begin{split}
            \partial_{y_i}A_t^N(x,y)&=\frac{1}{N}\sum_{k=1}^N \Bigg [
            \sigma(w_t^k\cdot x+b_t^k)\sigma'(w_t^k\cdot x+b_t^k)(w_t^k)_i  + {c_t^k}^2  \sigma'(w_t^k\cdot x+b_t^k) \sigma'(w_t^k\cdot y+b_t^k) x_i\\
            &\qquad +{c_t^k}^2  \sigma'(w_t^k\cdot x+b_t^k) \sigma''(w_t^k\cdot y+b_t^k) x \cdot y (w_t^k)_i \Bigg ].
        \end{split}
    \end{align}
    Similarly to in \eqref{ata0_1}, we split the difference between $\partial_{y_i}A_t^N(x,y)$ and $\partial_{y_i}A_0^N(x,y)$ into terms \begin{align}
        \begin{split}
            &|\partial_{y_i}A_t^N(x,y)-\partial_{y_i}A_0^N(x,y)| \leq \frac{1}{N}\sum_{k=1}^N \bigg [
            \Big|\sigma(w_t^k\cdot x+b_t^k)\sigma'(w_t^k\cdot x+b_t^k)(w_t^k)_i\\&\qquad\qquad
            -\sigma(w_0^k\cdot x+b_0^k)\sigma'(w_0^k\cdot x+b_0^k)(w_0^k)_i\Big|\\
            &\quad +\Big|{c_t^k}^2  \sigma'(w_t^k\cdot x+b_t^k) \sigma'(w_t^k\cdot y+b_t^k) x_i - {c_0^k}^2  \sigma'(w_0^k\cdot x+b_0^k) \sigma'(w_0^k\cdot y+b_0^k) x_i\Big|\\
            &\quad +\Big|{c_t^k}^2  \sigma'(w_t^k\cdot x+b_t^k) \sigma''(w_t^k\cdot y+b_t^k) x \cdot y (w_t^k)_i -{c_0^k}^2  \sigma'(w_0^k\cdot x+b_0^k) \sigma''(w_0^k\cdot y+b_0^k) x \cdot y (w_0^k)_i\Big|\bigg ]. \label{ata0_2}
        \end{split}
    \end{align}
    Applying the mean value theorem with respect to $w_t^i$ to each term on the RHS of (\ref{ata0_2}), using the fact that $\sigma \in C_b^4(\mathbb{R})$ and that $|c_t^i|< C$ for any $t \in [0,T]$, we have \begin{align}
        \begin{split}
            &|\partial_{y_i}A_t^N(x,y)-\partial_{y_i}A_0^N(x,y)| \leq \frac{C}{N}\sum_{k=1}^N \bigg [
            |(w_t^k)_i-(w_0^k)_i|+|(w_0^k)_i|(\|x\|_1+\|y\|_1)\|w_t^k-w_0^k\|_1\\
            &\quad + |c_t^i-c^i_0||x_i|+|x_i|(\|x\|_1+\|y\|_1)\|w_t^k-w_0^k\|_1+|(w_t^k)_i-(w_0^k)_i||x\cdot y|\\
            & \quad + |c_t^i-c^i_0||(w_0^k)^i||x \cdot y|+|(w_0^k)^i||x \cdot y|(\|x\|_1+\|y\|_1)||w_t^k-w_0^k||_1 + (1+|x \cdot y|)|b_t^k-b_0^k|\bigg]. \notag
        \end{split}
    \end{align}
    Therefore, by the boundedness of $|c_t^k-c_0^k|$, $|b_t^k-b_0^k|$, $|(w_t^k)_i-(w_0^k)_i|$ and $||w_t^k-w_0^k||_1$ we derive \begin{align}
\begin{split}
    &|\partial_{y_i}A_t^N(x,y)-\partial_{y_i} A_0^N(x,y)| \leq C N^{\beta-1} \bigg [ 1+ \frac{\sum_{k=1}^N |(w_0^k)_i|}{N}(|x\cdot y|+1)(1+\|x\|_1+\|y\|_1)\\&\quad +|x_i|(1+\|x\|_1+\|y\|_1)+|x\cdot y| \bigg ] \leq C N^{\delta+\beta-1}. \notag
\end{split}
\end{align}
By definition,
\begin{align}
    \begin{split}
            \partial^2_{y_i y_j}A_t^N(x,y)&=\frac{1}{N}\sum_{k=1}^N \Bigg [
            \sigma(w_t^k\cdot x+b_t^k)\sigma''(w_t^k\cdot x+b_t^k)(w_t^k)_i(w_t^k)_j \\ &         + {c_t^k}^2  \sigma'(w_t^k\cdot x+b_t^k) \sigma''(w_t^k\cdot y+b_t^k) x_i(w_t^k)_j
            \\
            &\quad + {c_t^k}^2  \sigma'(w_t^k\cdot x+b_t^k) \sigma''(w_t^k\cdot y+b_t^k) x_j(w_t^k)_i \\
            &\quad +{c_t^k}^2  \sigma'(w_t^k\cdot x+b_t^k) \sigma'''(w_t^k\cdot y+b_t^k) x \cdot y (w_t^k)_i(w_t^k)_j\bigg ]. \notag
        \end{split}
\end{align}
The difference between $\partial^2_{y_i y_j}A^N_t$ and $\partial^2_{y_i y_j}A^N_0$ can be split into terms 
\begin{align}
    \begin{split}
         &|\partial^2_{y_i y_j}A_t^N(x,y)-\partial^2_{y_i y_j}A_0^N(x,y)| \leq \frac{1}{N}\sum_{k=1}^N \bigg [
            \Big|\sigma(w_t^k\cdot x+b_t^k)\sigma''(w_t^k\cdot x+b_t^k)(w_t^k)_i(w_t^k)_k\\
            &\qquad 
            -\sigma(w_0^k\cdot x+b_0^k)\sigma''(w_0^k\cdot x+b_0^k)(w_0^k)_i(w_0^k)_j\Big|
 +\Big|{c_t^k}^2  \sigma'(w_t^k\cdot x+b_t^k) \sigma''(w_t^k\cdot y+b_t^k) x_i(w_t^k)_j \\
            &\qquad - {c_0^k}^2  \sigma'(w_0^k\cdot x+b_0^k) \sigma''(w_0^k\cdot y+b_0^k) x_i(w_0^k)_j\Big| +\Big|{c_t^k}^2  \sigma'(w_t^k\cdot x+b_t^k) \sigma''(w_t^k\cdot y+b_t^k) x_j(w_t^k)_i \\
            &\qquad - {c_0^k}^2  \sigma'(w_0^k\cdot x+b_0^k) \sigma''(w_0^k\cdot y+b_0^k) x_j(w_0^k)_i\Big| +\Big|{c_t^k}^2  \sigma'(w_t^k\cdot x+b_t^k) \sigma'''(w_t^k\cdot y+b_t^k) x \cdot y (w_t^k)_i(w_t^k)_j
            \\
            &\qquad -{c_0^k}^2  \sigma'(w_0^k\cdot x+b_0^k) \sigma'''(w_0^k\cdot y+b_0^k) x \cdot y (w_0^k)_i(w_0^k)_j\Big|\bigg ]. \label{ata0_3}
    \end{split}
\end{align}
We apply the mean value theorem to each term above, leading to 
\begin{align}
\begin{split}
    &|\partial^2_{y_i y_j}A_t^N(x,y)-\partial^2_{y_i y_j} A_0^N(x,y)|  \leq C N^{\beta-1} \bigg [ \frac{\sum_{k=1}^N |(w_0^k)_i||(w_0^k)_j|}{N}(|x \cdot y|+1)(\|x\|_1+\|y\|_1+1) \\
    &\quad + \frac{\sum_{k=1}^N |(w_0^k)_i|}{N}[(1+|x_j|)(1+\|x\|_1+\|y\|_1)+|x \cdot y|] \\
    &\quad+ \frac{\sum_{k=1}^N |(w_0^k)_j|}{N}[(1+|x_i|)(1+\|x\|_1+\|y\|_1)+|x \cdot y|] + 1+|x_i|+|x_j|+ \cdot y| \bigg ] \leq C N^{\delta+\beta-1}. \notag
\end{split}
\end{align}
\end{proof}
\subsection{Proof of Lemma \ref{lemma33}}
\begin{proof}
   Taking squares and then expectations on both sides of inequalities (\ref{39}), (\ref{40}), and (\ref{41}) gives 
   \begin{align}
       \begin{split}
           \mathbb{E}_{c_0,w_0,b_0}\big[|A_t^N(x,y)-A_0^N(x,y)|^2\big] \leq CN^{2(\delta+\beta-1)},\\
           \mathbb{E}_{c_0,w_0,b_0}\big[|\partial_{y_k}A_t^N(x,y)-\partial_{y_k}A_0^N(x,y)|^2\big] \leq CN^{2(\delta+\beta-1)},\\
           \mathbb{E}_{c_0,w_0,b_0}\big[|\partial^2_{y_k y_l}A_t^N(x,y)-\partial^2_{y_k y_l}A_0^N(x,y)|^2\big] \leq CN^{2(\delta+\beta-1)}.
       \end{split}
   \end{align}
   Summing these inequalities for all $k,l\in \{1,2,..,n\}$ and applying the Cauchy--Schwarz inequality gives the result.
\end{proof}
\subsection{Proof of Lemma \ref{a0a_lemma}}
\begin{proof}
    First note that \begin{align}
    \begin{split}
         \mathbb{E}_{c_0,w_0,b_0}[A_0^N(x,y)]&= \mathbb{E}_{c_0,w_0,b_0}\bigg[\frac{1}{N} \sum_{k=1}^N \sigma(w_0^k \cdot x + b_0^k)\sigma(w_0^k \cdot y + b_0^k)\\&\quad+{c_0^k}^2 x \cdot y \sigma'(w_0^k \cdot x + b_0^k)\sigma'(w_0^k \cdot y + b_0^k)\bigg] = A(x,y).
    \end{split}
    \end{align}
    We can also calculate
    \begin{align}
    \begin{split}
        &\mathbb{E}_{c_0,w_0,b_0} \big[|A_0^N(x,y)-A(x,y)|^2\big] =\text{Var}[A_0^N(x,y)] =\frac{1}{N}\text{Var}[U(x;c,w,b)\cdot U(y;c,w,b)] \\ &\leq \frac{1}{N}\mathbb{E}_{c,w,b}|U(x;c,w,b)\cdot U(y;c,w,b)|^2 \leq \frac{C}{N}.
    \end{split}
    \end{align}
For the partial derivatives, for any fixed $x,y \in \Omega$, 
    \begin{align}\begin{split}
         &|\partial_{y_i} (\sigma(w\cdot x +b)\sigma(w \cdot y + b)+c^2 \sigma'(w\cdot x +b)\sigma'(w \cdot y +b) x \cdot y)|\\
         &=|\sigma(w \cdot x +b)\sigma'(w \cdot y +b)w_i\\
         &\quad+c^2\sigma'(w \cdot x+b) \sigma'(w \cdot y +b)x_i+c^2  x\cdot y\sigma'(w \cdot x+b) \sigma''(w \cdot y +b)w_i|
    \end{split}
    \end{align}
    and
    \begin{align}\begin{split}
         &|\partial^2_{y_i y_j} (\sigma(w\cdot x +b)\sigma(w \cdot y + b)+c^2 \sigma'(w\cdot x +b)\sigma'(w \cdot y +b) x \cdot y)|\\
         &=|\sigma(w \cdot x +b)\sigma''(w \cdot y +b)w_i w_j+c^2\sigma'(w \cdot x+b) \sigma''(w \cdot y +b)(x_i w_j+x_j w_i)\\&\quad +c^2  x\cdot y\sigma'(w \cdot x+b) \sigma'''(w \cdot y +b)w_i w_j|.
    \end{split}
    \end{align}
    These are locally bounded around $(x,y)\in \Omega^2$, so by Leibniz's integral rule, we swap the expectation and the differential operator \begin{align}
    \begin{split}
        &\partial_{y_i}A(x,y) =\partial_{y_i} \mathbb{E}_{c,w,b}[\sigma(w\cdot x +b)\sigma(w \cdot y + b)+c^2 \sigma'(w\cdot x +b)\sigma'(w \cdot y +b) x \cdot y]
        \\
        &=\mathbb{E}_{c,w,b}[\partial_{y_i}(\sigma(w\cdot x +b)\sigma(w \cdot y + b)+c^2 \sigma'(w\cdot x +b)\sigma'(w \cdot y +b) x \cdot y)]\\
        &=\mathbb{E}_{c_0,w_0,b_0}[\partial_{y_i}A_0^N(x,y)],
    \end{split}
    \end{align}
    and 
    \begin{align}
    \begin{split}
        &\partial^2_{y_i y_j}A(x,y) =\partial^2_{y_i y_j} \mathbb{E}_{c,w,b}[\sigma(w\cdot x +b)\sigma(w \cdot y + b)+c^2 \sigma'(w\cdot x +b)\sigma'(w \cdot y +b) x \cdot y]
        \\
        &=\mathbb{E}_{c,w,b}[\partial^2_{y_i y_j}(\sigma(w\cdot x +b)\sigma(w \cdot y + b)+c^2 \sigma'(w\cdot x +b)\sigma'(w \cdot y +b) x \cdot y)]\\
        &=\mathbb{E}_{c_0,w_0,b_0}[\partial^2_{y_i y_j}A_0^N(x,y)].
    \end{split}
    \end{align}
    Therefore, we can bound the first-derivative variance $\text{Var}[\partial_{y_i}A_0^N(x,y)]$:
    \begin{align}
    \begin{split}
        &\mathbb{E}_{c_0,w_0,b_0} |\partial_{y_i}A_0^N(x,y)-\partial_{y_i}A(x,y)|^2 =\frac{1}{N}\text{Var}[\partial_{y_i}(U(x;c,w,b)\cdot U(y;c,w,b))]\\
        &\leq \frac{1}{N}\mathbb{E}_{c,w,b}|\partial_{y_i}(U(x;c,w,b)\cdot U(y;c,w,b))|^2 \leq \frac{1}{N} \mathbb{E}_{c,w,b}[(C|w_i|+C|x_i|+C|x \cdot y||w_i|)^2] \leq \frac{C}{N}. \notag 
    \end{split}
    \end{align}
    and similarly the second derivative variance $\text{Var}[\partial^2_{y_i y_j}A_0^N(x,y)]$:
    \begin{align}
    \begin{split}
        &\mathbb{E}_{c_0,w_0,b_0} |\partial^2_{y_i y_j}A_0^N(x,y)-\partial^2_{y_i y_j}A(x,y)|^2 =\frac{1}{N}\text{Var}[\partial^2_{y_i y_j}(U(x;c,w,b)\cdot U(y;c,w,b))]\\
        &\leq \frac{1}{N}\mathbb{E}_{c,w,b}|\partial^2_{y_i y_j}(U(x;c,w,b)\cdot U(y;c,w,b))|^2 \\
        &\leq \frac{1}{N} \mathbb{E}_{c,w,b}[(C|w_i w_j|+C|x_iw_j|+C|x_j w_i|+C|x \cdot y||w_i w_j|)^2] \leq \frac{C}{N}. \notag
    \end{split}
    \end{align}
\end{proof}
\subsection{Existence of ODE solutions}\label{lem:ODEexist}
\begin{lemma}
Under the assumptions of Theorem \ref{thm:Qconvergence}, for any $q\in \mathcal{H}^2$, and any initial value $v_0\in L^2_\eta$, the ODE
\begin{align}
    \frac{dv_t}{dt} = \mathcal{L}(\mathcal{B}v_t + q)
\end{align}
admits a unique solution $v$ taking values in $L^2_\eta$.
\end{lemma}
\begin{proof}
Consider the process $w_t = \eta v_t$. This has initial value $w_0= \eta v_0\in L^2$, and satisfies the dynamics 
\begin{equation}\label{eq:wLipconstruction}
    \frac{dw_t}{dt}  =\eta\mathcal{L}(\mathcal{B}(w_t/\eta) + q) =:G(w_t).
\end{equation} From Lemma \ref{lem:Blip1}, the map $w\mapsto \mathcal{B}(w/\eta)$ is uniformly Lipschitz as a map $L^2\to \mathcal{H}^2$. By Assumption \ref{assume_1}, $\mathcal{L}$ is uniformly Lipschitz as a map $\mathcal{H}^2 \to L^2$. By Assumption \ref{auxiliaryfunction}, $\eta$ is bounded on $\Omega$. Therefore, $G:L^2\to L^2$ is uniformly  Lipschitz.

From the Picard--Lindel\"of theorem (see, for example, \cite[Theorem 2.2.1]{kolokoltsov2019differential}), we know that \eqref{eq:wLipconstruction} admits a unique solution with $w_t \in L^2$ for all $t\ge 0$. Using this, together with the fact $\eta>0$ on $\Omega$, we see that there is a unique $v_t = w_t/\eta$ satisfying the original equation, with values in $L^2_\eta$.
\end{proof}

\subsection{Monotonicity for linear elliptic PDE}\label{appMonotone}

In this subsection, we give a brief proof of monotonicity of a class of elliptic linear differential operators, with particular choices of sampling measure $\mu$.  Our results in this appendix are not exhaustive, but demonstrate some of the flexibility of the equations we have considered.

\begin{lemma}\label{semilinearProbDissipative}
Consider an autonomous Markov process $X$ satisfying the SDE
\begin{equation}\label{Markovprocess}
dX_t = \nu(X_t) dt+ \sigma(X_t) dW_t
\end{equation}
where $\nu:\overline\Omega\to \mathbb{R}^n$ and $\sigma:\overline \Omega\to \mathbb{R}^{n\times n}$ and $W$ is an $\mathbb{R}^n$-valued standard Brownian motion. Suppose further that $\nu$ and $\sigma$ both vanish outside $\Omega$. For an absolutely continuous measure $\mu$ on $\Omega$, with density $f_\mu$, let $\nu$, $\sigma$ and $f_\mu$ be sufficiently smooth that 
\begin{equation}\label{gamma*def}
    \gamma^* = \sup_{x\in \Omega}\Bigg\{\frac{-\sum_i\frac{\partial}{\partial x_i}[\nu_i(x) f_\mu(x)] + \sum_{i,j} \frac{\partial^2}{\partial x_i\partial x_j}[(\sigma\sigma^\top)_{ij}(x) f_{\mu}(x)]}{f_\mu(x)}\Bigg\}<\infty.
\end{equation}
Then the generator of $X$ satisfies, for all $u\in \mathcal{H}^2$, 
\begin{equation}\label{dissipativegenerator}
\Big \langle u, \sum_i\nu_i \frac{\partial u}{\partial x_i} + \frac{1}{2}\sum_{ij} (\sigma\sigma^\top)_{ij}\frac{\partial^2 u}{\partial x_i\partial x_j}\Big\rangle \leq \frac{\gamma^*}{2} \|u\|,\end{equation}
where the inner product and norm are both taken in $L^2(\Omega, \mu)$.
\end{lemma}
\begin{proof}
Consider a copy of $X$ initialized with $X_0\sim \mu$. The Fokker--Plank equation states that the density of $X_0$ satisfies $f(x,0) = f_\mu$ and
\begin{equation}
    \frac{\partial}{\partial t} f(x,t) = -\sum_i\frac{\partial}{\partial x_i}[\nu_i(x) f(x,t)] + \sum_{i,j} \frac{\partial^2}{\partial x_i\partial x_j}[(\sigma\sigma^\top)_{ij}(x) f(x,t)]
\end{equation}
and hence, using the definition of $\gamma^*$,
\begin{equation}
    \frac{\partial}{\partial t}(e^{-\gamma^*t} f(x,t)) \leq 0.
\end{equation}
As $f(x,t)$ is positive, we can define an operator
\begin{equation}
   T_tu(x) := \mathbb{E}\Big[u(X_t^x) e^{-\frac{\gamma^*}{2}t}\Big|X_0=x\Big]
   \end{equation}
   which satisfies, by Jensen's inequality,
 \begin{equation}
\|T_t u\|_{L^2(\mu)} \leq  \int_\Omega u^2(t)e^{-\gamma^* t} f(t,x) dx  \leq \int_\Omega u^2(t) f_\mu(x) dx = \|u\|^2_{L^2(\mu)}
\end{equation}
so the operator $T_t$ is a contraction. It is easy to verify that $T_t$ is also a strongly continuous semigroup, and its generator is given by $u\mapsto-\frac{\gamma^*}{2} u + \mathcal{A}u$, where $\mathcal{A}$ is the generator of $X$. By the Lumer--Phillips theorem (\cite{lumer1961}), we know that the generator of $T_t$ is dissipative on $L^2(\Omega, \mu)$, that is
\begin{equation}
\Big \langle u, -\frac{\gamma^*}{2}u+ \mathcal{A}u\Big\rangle= \Big \langle u, -\frac{\gamma^*}{2}u+ \sum_i\nu_i \frac{\partial u}{\partial x_i} + \frac{1}{2}\sum_{ij} (\sigma\sigma^\top)_{ij}\frac{\partial^2 u}{\partial x_i\partial x_j}\Big\rangle \leq 0.\end{equation}
Rearrangement yields the result.
\end{proof}

\begin{remark}
This result immediately shows that if $\mu$ is a stationary distribution for the Markov process, then $\gamma^*=0$ and the generator of $X$ is automatically $L^2$-dissipative. (This result is known, see, for example, \cite[Chapter 20]{kallenberg}.) This suggests that this will often be a wise choice for $\mu$ in linear problems, if it is known explicitly.
\end{remark}

By manipulating the choice of measure $\mu$, or equivalently setting $\mu$ equal to Lebesgue measure and considering the equivalent PDE 
\begin{equation}
    \tilde{\mathcal{L}} u :=f(x) \mathcal{L}u=0,
\end{equation} 
this result can be leveraged usefully, provided the drift of $X$ does not grow quickly. Similar results for other bounds on $b$ can also be obtained.

\begin{lemma}
Consider a process $X$ as in Lemma \ref{semilinearProbDissipative}. Suppose that $\sigma$ is constant and there exist bounded functions $g_i:\mathbb{R}\to \mathbb{R}$ for $i\leq n$ such that, for all $x\in \Omega$,
\begin{equation}
\begin{split}\label{gcondition}
\gamma &\geq -\sum_i\Big(\frac{\partial}{\partial x_i} \nu_i(x) +g(x_i) \nu_i(x)\Big)+ \sum_{ij} g_i(x_i)g_j(x_j)(\sigma\sigma^\top)_{ij}.
\end{split}
\end{equation}
Then the generator of $X$ satisfies $\langle u, \mathcal{A} u\rangle \leq \frac{\gamma}{2}\|u\|$, where $\mu$ is the measure on $\Omega$ with Lebesgue density
\[f_\mu = C \exp\Big(\sum_i \int_0^{x_i} g_i(\zeta)d\zeta\Big),\]
for $C$ a normalizing constant. In particular, this implies that the operator
\begin{equation}\mathcal{L}u = r-\gamma u + \mathcal{A}u\end{equation}
is strongly monotone, in the sense of Assumption \ref{Lmonotone}.
\end{lemma}
\begin{proof}
We calculate
\begin{equation}
\begin{split}
    &-\sum_i\frac{\partial}{\partial x_i}[\nu_i(x) f_\mu(x)] + \sum_{i,j} \frac{\partial^2}{\partial x_i\partial x_j}[(\sigma\sigma^\top)_{ij} f_{\mu}(x)]\\
    &=-\sum_i\Big(\frac{\partial}{\partial x_i}\nu_i(x) f_\mu(x) + g_i(x_i)\nu_i(x) f_\mu(x)\big) + \sum_{i,j}g(x_i)g(x_j)(\sigma\sigma^\top)_{ij}f_\mu(x)\\
    &\leq \gamma f_\mu(x).
    \end{split}
    \end{equation}
Therefore, in \eqref{gamma*def} we have $\gamma^*\leq \gamma $, and the result follows from \eqref{dissipativegenerator}.
\end{proof}

\begin{remark}
In one dimension, \eqref{gcondition} is clearly satisfied whenever we know that  $4\sigma^2(\gamma+\nu')\geq -\nu^2$, by taking $g(x) = -\nu(x)/(2\sigma^2)$.
\end{remark}

\subsection{Configuration of the numerical test cases}\label{config}
\subsubsection{Table of hyper parameters}
\begin{tabular}{|c|c|c|c|c|c|c|}
\hline
  Dimension   & Method & Layer & Units & Activation & Optimizer & Numer of MC samples \\
  \hline
  1   & Q-PDE & 1 & 64 & Sigmoid & ADAM & $l_{MC}$=1k, $u_{MC}$=2k  \\
  \hline
  20   & Q-PDE & 1 & 256 & Sigmoid & ADAM & $l_{MC}$=2k, $u_{MC}$=10k  \\
  \hline
  20   & DGM & 1 & 256 & Sigmoid & ADAM & $l_{MC}$=2k, $u_{MC}$=10k  \\
  \hline
  
\end{tabular}

\subsubsection{Initialization of neural networks}
Parameters of the single-layer net $S_0$ are randomly sampled: $c_0^i$ are i.i.d sampled from uniform distribution $\mathcal{U}[-1,1]$; $w_0^i$ and $b_0^i$ are i.i.d sampled from Gaussian distribution $\mathcal{N}(0,I_d)$ and $\mathcal{N}(0,1)$ where $I_d$ is the identity matrix of dimension $d$.
\subsubsection{Learning process}
We use $Q_t:=S_t\cdot \eta$ as the approximator, where $\eta(x):= 1-\|x\|^2$.
We apply the built-in ADAM optimizer in our test with initial learning rate $l_0=0.5$, and the learning rate decays as $l_t=l_0/(1+t/200)$. In each step, we sample MC samples for gradient estimate. As a larger number of MC sample points reduces the random error, we linearly increase the number $M_t$ of MC points to be sampled  at each step as $M_t=\mathrm{round}(l_{MC}+(u_{MC}-l_{MC})t/T)$, where $T$ is the terminal number of training steps.

\vskip 0.2in
\bibliography{QPDE}

\end{document}